\documentclass[10pt]{article}
\usepackage[a4paper]{geometry}

\usepackage[T1]{fontenc}
\usepackage[utf8]{inputenc}
\usepackage[british]{babel}
\usepackage[dvipsnames,svgnames,table]{xcolor}
\usepackage{lmodern}
\usepackage{csquotes}
\usepackage{amsmath,amsthm,amsfonts,amssymb,mathtools}
\usepackage{enumitem}
\usepackage{thmtools}
\usepackage{commath}
\usepackage{nicefrac}
\usepackage{stmaryrd}
\usepackage{etoolbox}
\usepackage{booktabs}
\usepackage{multirow}
\usepackage{hyperref}
\usepackage[capitalize]{cleveref}
\usepackage{graphicx}
\usepackage{subcaption}
\usepackage{todonotes}
\usepackage{tabularray}
\usepackage{tikz}
\usepackage{cite}

\usepackage{tikz-3dplot}
\usetikzlibrary{positioning, shapes, arrows, calc, arrows.meta, patterns, shadows, trees, intersections}

\DeclareMathOperator*{\argmin}{arg\,min}

\newtheorem{theorem}{Theorem}[section] %
\newtheorem{lemma}[theorem]{Lemma}

\newtheorem{remark}[theorem]{Remark}

\newtheorem{corollary}[theorem]{Corollary}
\newtheorem{assumption}[theorem]{Assumption}

\newcommand{\diam}{\mathrm{diam}}

\newcommand{\Th}{{\mathcal{T}_h}}
\newcommand{\elm}{T}                %
\newcommand{\Fh}{{\mathcal{F}_h}}
\newcommand{\Fhi}{{\mathcal{F}_h^{\text{i}}}}
\newcommand{\Fhb}{{\mathcal{F}_h^{\text{b}}}}
\newcommand{\BDM}{\mathbb{B}\mathbb{D}\mathbb{M}}

\newcommand{\dom}{{\Omega}}

\newenvironment{eqs} %
 { \begin{equation} \begin{aligned} } %
 { \end{aligned} \end{equation} \ignorespacesafterend } %

\renewcommand{\norm}[1]{\lVert #1 \rVert}
\newcommand{\dgnorm}[1]{\norm{#1}_{\IX}}
\newcommand{\dgsnorm}[1]{\norm{#1}_{\IX,\ast}}
\newcommand{\vdgnorm}[1]{\norm{#1}_{1,h}}
\newcommand{\vdgsnorm}[1]{\norm{#1}_{1,h,\ast}}

\newcommand{\stablocal}[1]{\alpha_{\scalebox{0.6}{$K_{#1}$}}}
\newcommand{\contlocal}[1]{\beta_{\scalebox{0.6}{$K_{#1}$}}}
\newcommand{\staboseen}[1]{\alpha_{\scalebox{0.6}{$k_{#1}$}}}
\newcommand{\contoseen}[1]{\beta_{\scalebox{0.6}{$k_{#1}$}}}
\newcommand{\staboseenr}[1]{\alpha_{\scalebox{0.6}{$\underline{k}_{#1}$}}}
\newcommand{\contoseenr}[1]{\beta_{\scalebox{0.6}{$\underline{k}_{#1}$}}}
\newcommand{\contcross}[1]{\beta_{\scalebox{0.6}{$\times_{#1}$}}}

\newcommand{\conttrefftz}[1]{\beta_{\scalebox{0.6}{$c_{#1}$}}}
\newcommand{\stabtrefftz}[1]{\alpha_{\scalebox{0.6}{$c_{#1}$}}}
\newcommand{\stabtrefftzr}[1]{\alpha_{\scalebox{0.6}{$\underline{c}_{#1}$}}}
\newcommand{\conttrefftzr}[1]{\beta_{\scalebox{0.6}{$\underline{c}_{#1}$}}}

\newcommand{\contconv}{\beta_{\scalebox{0.6}{${a}^{\texttt{c}}$}}}   %
\newcommand{\contvis}{\beta_{\scalebox{0.6}{${a}^{\texttt{d}}$}}}   %
\newcommand{\contdiv}{\beta_{\scalebox{0.6}{$b$}}}                    %
\newcommand{\lambdadg}{\lambda_{\texttt{DG}}}

\DeclareMathOperator{\id}{id}

\DeclareMathOperator{\Div}{\mathrm{div}}%
\DeclareMathOperator{\Grad}{\nabla}%

\DeclareMathOperator\Span{span}

\newcommand\restr[2]{{\left.\kern-\nulldelimiterspace #1 \right|_{#2} }}

\newcommand{\avg}[1]{\{\!\!\{ #1\}\!\!\}}
\newcommand{\jmp}[1]{[\![#1]\!]}

\newcommand{\up}{\xi}    %
\newcommand{\vq}{\eta}   %
\newcommand{\rs}{\sigma} %
\newcommand{\z}{\zeta}   %

\newcommand{\IB}{\mathbb{B}}

\newcommand{\IL}{\mathbb{L}}
\newcommand{\ILu}{\IL^{\!\!u}}   %
\newcommand{\ILp}{\IL^{\!\!p}}   %

\newcommand{\IP}{\mathbb{P}}
\newcommand{\IQ}{\mathbb{Q}}
\newcommand{\IR}{\mathbb{R}}

\newcommand{\IT}{\mathbb{T}}

\newcommand{\IW}{\mathbb{W}}
\newcommand{\IX}{\mathbb{X}}
\newcommand{\IXreg}{\mathbb{X}^{\texttt{reg}}}

\newcommand{\IZ}{\mathbb{Z}}

\newcommand{\TR}{\underline{\mathbb{T}}}
\newcommand{\QR}{\underline{\mathbb{Q}}}
\newcommand{\LR}{\underline{\mathbb{L}}}
\newcommand{\XR}{\underline{\mathbb{X}}}

\newcommand{\ZR}{\underline{\mathbb{Z}}}
\newcommand{\XRreg}{\underline{\mathbb{X}}^{\texttt{reg}}}
\newcommand{\ur}{\underline{u}}
\newcommand{\vr}{\underline{v}}
\newcommand{\rr}{\underline{r}}
\newcommand{\zr}{\underline{z}}
\newcommand{\Plift}[1]{\mathcal{P}_{\!#1}} %

\newcommand{\blfvis}{a_h^{\!\texttt{d}}}             %
\newcommand{\blfdiv}{b_h}                             %
\newcommand{\blfconv}[1]{a_{h,#1}^{\!\texttt{c}}}    %
\newcommand{\blfoseen}[1]{k_{h,#1}}                  %
\newcommand{\blftrefftz}[1]{c_{h,#1}}                      %
\newcommand{\blfoseenr}[1]{\underline{k}_{h,#1}}     %
\newcommand{\blftrefftzr}[1]{\underline{c}_{h,#1}}         %

\newcommand{\opconvelem}[1]{A^{\!\texttt{c}}_{\elm,#1}}   %
\newcommand{\opconv}[1]{A^{\!\texttt{c}}_{#1}}         %
\newcommand{\oposeenelem}[1]{K_{\elm,#1}}                    %
\newcommand{\oposeen}[1]{K_{#1}}                          %
\newcommand{\oposeenrelem}[1]{\underline{K}_{\elm,#1}}       %
\newcommand{\oposeenr}[1]{\underline{K}_{#1}}             %
\newcommand{\opcd}[1]{A_{#1}^{\!\texttt{cd}}}        %
\newcommand{\rmod}[1]{\underline{#1}}

\newcommand{\calL}{\mathcal{L}}

\newcommand{\calP}{\mathcal{P}}

\newcommand{\calT}{\mathcal{T}}

\usepackage{pgfplots}
\usepgfplotslibrary{colorbrewer,groupplots}
\pgfplotsset{
  compat=1.15,
  cycle list/Set1-5,
  title style = {font = \small},
  legend style = {font = \small},
  label style = {font = \footnotesize},
  tick label style = {font = \footnotesize},
  PPk/.style={mark=square*},
  TTk/.style={mark=*},
  TTkw/.style={mark=triangle*},
  TTkemb/.style={mark=diamond*},
}
\pgfplotsset{
    discard if not/.style 2 args={
        x filter/.append code={
            \edef\tempa{\thisrow{#1}}
            \edef\tempb{#2}
            \ifx\tempa\tempb
            \else
                
            \fi
        }
    }
}
\pgfplotsset{
    discard if/.style 2 args={
        x filter/.append code={
            \edef\tempa{\thisrow{#1}}
            \edef\tempb{#2}
            \ifx\tempa\tempb
                
            \fi
        }
    }
}

\pgfplotscreateplotcyclelist{paulcolors3}{%
violet,every mark/.append style={solid,fill=violet},mark=square*,very thick,mark size=3pt\\%
teal,every mark/.append style={solid,fill=teal},mark=*,very thick,mark size=2.8pt\\%
orange,every mark/.append style={solid,fill=orange},mark=diamond*,very thick,mark size=2pt\\%
violet,densely dashed,every mark/.append style={solid,fill=violet},mark=square*,very thick,mark size=3pt\\%
teal,densely dashed,every mark/.append style={solid,fill=teal},mark=*,very thick,mark size=2.8pt\\%
orange,densely dashed,every mark/.append style={solid,fill=orange},mark=diamond*,very thick,mark size=2pt\\%
violet,dash dot,every mark/.append style={solid,fill=violet},mark=square*,very thick,mark size=3pt\\%
teal,dash dot,every mark/.append style={solid,fill=teal},mark=*,very thick,mark size=2.8pt\\%
orange,dash dot,every mark/.append style={solid,fill=orange},mark=diamond*,very thick,mark size=2pt\\%
}

\pgfplotscreateplotcyclelist{paulcolors4}{%
violet,every mark/.append style={solid,fill=violet},mark=square*,very thick,mark size=3pt\\%
teal,every mark/.append style={solid,fill=teal},mark=*,very thick,mark size=3pt\\%
orange,every mark/.append style={solid,fill=orange},mark=diamond*,very thick,mark size=3pt\\%
cyan,every mark/.append style={solid,fill=cyan},mark=star,very thick,mark size=3pt\\%
violet,densely dashed,every mark/.append style={solid,fill=violet},mark=square*,very thick,mark size=3pt\\%
teal,densely dashed,every mark/.append style={solid,fill=teal},mark=*,very thick,mark size=2.8pt\\%
orange,densely dashed,every mark/.append style={solid,fill=orange},mark=diamond*,very thick,mark size=2pt\\%
cyan,densely dashed,every mark/.append style={solid,fill=cyan},mark=star,very thick,mark size=3pt\\%
violet,dash dot,every mark/.append style={solid,fill=violet},mark=square*,very thick,mark size=3pt\\%
teal,dash dot,every mark/.append style={solid,fill=teal},mark=*,very thick,mark size=2.8pt\\%
orange,dash dot,every mark/.append style={solid,fill=orange},mark=diamond*,very thick,mark size=2pt\\%
cyan,dash dot,every mark/.append style={solid,fill=cyan},mark=star,very thick,mark size=3pt\\%
}

\pgfplotscreateplotcyclelist{paulcolors2}{%
{violet,every mark/.append style={solid,fill=violet},mark=square*,very thick,mark size=3pt}\\%
{violet,densely dashed,every mark/.append style={solid,fill=violet},mark=square*,very thick,mark size=3pt}\\%
{teal,every mark/.append style={solid,fill=teal},mark=*,very thick,mark size=3pt}\\%
{teal,densely dashed,every mark/.append style={solid,fill=teal},mark=*,very thick,mark size=2.8pt}\\%
}

\usepackage{pgfplotstable} 
\makeatletter
\pgfplotstableset{
	discard if not/.style 2 args={
	row predicate/.append code={
	\def\pgfplotstable@loc@TMPd{\pgfplotstablegetelem{##1}{#1}\of}
	\expandafter\pgfplotstable@loc@TMPd\pgfplotstablename
	\edef\tempa{\pgfplotsretval}
	\edef\tempb{#2}
	\ifx\tempa\tempb
	\else
	\pgfplotstableuserowfalse
	\fi
	}
	}
}
\makeatother

\usepackage[plain]{algorithm}
\usepackage{algpseudocodex}

\title
{
Embedded Trefftz DG method for steady Navier--Stokes flow
\\[4pt]\large Part I: Oseen linearization}
\author{%
  Paul Stocker\thanks{Faculty of Mathematics, University of Vienna,
  Oskar-Morgenstern-Platz 1, 1090 Vienna, Austria.}
  \and
  Igor Voulis\thanks{Institute for Numerical and Applied Mathematics,
  University of G\"ottingen, Lotzestr.\ 16--18, 37083 G\"ottingen, Germany.}
  \and
  Christoph Lehrenfeld\footnotemark[2]
  \and
  Philip L.\ Lederer\thanks{Department of Mathematics, University of Hamburg,
  Bundesstra{\ss}e 55, 20146 Hamburg, Germany.}
}

\date{}

\begin{document}
\maketitle

\begin{abstract}
    We develop an embedded Trefftz-DG method for the Oseen problem and prove a complete stability and quasi-optimality theory in standard DG norms.
    The key ingredient is a construction of a suitable local complement space to the Trefftz space, on which the Oseen operator is stably invertible.
    We also derive a reduced formulation of the method, the resulting system is posed in terms of the velocity unknown only, a crucial step in the analysis especially for the nonlinear Navier--Stokes problem in Part~II.

\medskip
\noindent\textbf{Keywords}: 
embedded Trefftz methods, discontinuous Galerkin methods, Oseen problem

\medskip
\noindent\textbf{MSC2020}:
65N30, %
65N12, %
65N15, %
76D07, %
76M10 %
\end{abstract}

\section{Introduction}\label{sec:intro}

The present paper is the first part of a two-part work on embedded Trefftz discontinuous Galerkin methods for time-harmonic incompressible flow problems.
In this first part we consider the steady Oseen problem, which forms the linear building block for the Navier--Stokes analysis developed in the second part.
A very detailed comparison with other methods in terms of computational cost is given in \cite{LLS_NM_2024} for the Stokes problem, and the exact same comparison applies to the Oseen setting as well.

Trefftz methods, dating back to Trefftz \cite{trefftz1926}, are based on approximation spaces contained in the kernel of the underlying differential operator.
By building the PDE locally into the discrete space, they can deliver high approximation quality with substantially fewer degrees of freedom than standard polynomial methods, and are therefore particularly appealing in combination with discontinuous Galerkin (DG) discretizations.
Trefftz-DG methods have by now been developed for a broad range of problems; see, for instance, \cite{HMPS14,TrefftzSurvey,EKSW15,Huttunen,GMPS_AML_2023,21M1426079,mope18,bgl2016,bcds20,SpaceTimeTDG,KSTW2014,tpwave}.
For a detailed comparison to other methods on general meshes, see \cite{LSZ_PAMM_2024}.
At the same time, classical Trefftz constructions remain most natural for homogeneous linear equations with piecewise-constant coefficients, while variable coefficients and nonzero right-hand sides are well known to be significantly more challenging.

The embedded Trefftz-DG approach introduced in \cite{lozinski19,LS_CMAME_2016} avoids the explicit construction of local Trefftz bases.
Instead, one starts from an underlying polynomial DG space and imposes the PDE only through local projection constraints, thereby selecting a PDE-adapted subspace inside the full DG space.
This point of view makes Trefftz techniques accessible for variable-coefficient operators and inhomogeneous right-hand sides in a generic way.
The approach has already been investigated for a wide range of problems \cite{SV_ARXIV_2026, HLSW_M2AN_2022, M25, S23, H24, PIMPS_ARXIV_2026, GPS_JSC_2025}.
Most importantly for the present work is the study of the Stokes problem presented in \cite{LLS_NM_2024}, which is the direct predecessor of the present paper. 
There, higher-order pressure components are locally tied to the velocity unknowns, so that only piecewise-constant pressures remain globally coupled, and full stability and quasi-optimality were established.
A unified abstract local-global framework for embedded Trefftz methods was subsequently formulated in \cite{LLSV_ARXIV_2024}.

The extension from Stokes to Oseen is, however, far from formal.
The convection term introduces variable coefficients through the prescribed convection field, so the constant-coefficient arguments used in the Stokes analysis no longer apply directly.
In particular, the approximation arguments in \cite{LLS_NM_2024} rely on Taylor-polynomial constructions on affinely shifted Trefftz spaces, and these are tied to the constant-coefficient structure of the Stokes operator.
For the Oseen operator, a different route is needed.
The key idea of the present paper is therefore to exploit the abstract local-global splitting from \cite{LLSV_ARXIV_2024} and to construct suitable local complement spaces on which the Oseen operator is stably invertible.
This leads naturally to a reduced formulation for the velocity unknown, in which the higher-order pressure components are again locally determined by the velocity and only the piecewise-constant pressure component remains globally coupled.
The resulting analysis removes the saddle-point structure at the critical local step and allows us to transfer coercivity from the Stokes setting to reduced Oseen-Trefftz spaces under a suitable smallness condition on the convection field.

At the DG level, our construction is also connected to the discrete functional analysis framework of Di~Pietro and Ern \cite{DE10,DiPietroErn}, in particular to the treatment of convective terms and to energy-stable DG formulations for incompressible flow problems.
More broadly, the Oseen problem sits at the interface between classical DG discretizations of viscous incompressible flows and Trefftz-type model reduction ideas.
While Trefftz techniques for Stokes problems do exist, they are much less developed than for wave propagation problems; see, for instance, the spectral-type constructions in \cite{POITOU2000561,Bouberbachene}, the spectral method of \cite{LifitsQTSM}, and the collocation Trefftz approach of \cite{21710}.
As well as approximations via Solenoidal functions, i.e. element-wise divergence-free, a partial requirement of the PDE, in \cite{montlaur2010discontinuous, montlaur08}.
Another important and complementary strategy for reducing globally coupled degrees of freedom is provided by hybridizable DG methods \cite{cockburn2009unified}.
However, hybridization and Trefftz reduction act in rather different ways, and their benefits are to a large extent competing rather than cumulative.

\paragraph{Contributions.}
This paper extends the embedded Trefftz-DG methodology from the Stokes problem \cite{LLS_NM_2024} to the steady Oseen equations and provides the linear foundation for the nonlinear Navier--Stokes analysis in Part~II \cite{SVLL2_ARXIV_2026}.
Compared with the Stokes predecessor, the present work contains three main new ingredients:
\begin{itemize}
    \item We construct suitable local complement spaces $\IL(T)$ for the Oseen operator.
    In contrast to the Stokes case, where the construction relied on constant-coefficient Taylor-polynomial arguments, the convection term leads to a variable-coefficient local operator and requires a more delicate construction.

    \item We treat inhomogeneous right-hand sides through a local problem built directly into the formulation, rather than through a correction argument as in the Stokes analysis.
    This is better adapted to the Oseen setting and fits naturally into the abstract local-global splitting framework of \cite{LLSV_ARXIV_2024}.

    \item We derive a reduced formulation for the velocity unknown and prove reduced inf--sup stability for the full Oseen operator.
    More precisely, after eliminating the higher-order pressure components locally, we establish stability of the reduced Oseen problem and transfer it to the full embedded Trefftz-DG discretization, leading to well-posedness and quasi-optimality in standard DG norms.
\end{itemize}
These linear results provide the key analytical groundwork for the nonlinear Navier--Stokes theory developed in Part~II, \cite{SVLL2_ARXIV_2026}.

\paragraph{Outline.}
In \cref{sec:model} we introduce the steady Oseen problem as the linearization of the Navier--Stokes equations.
The embedded Trefftz-DG discretization and the local Trefftz constraint are presented in \cref{sec:trefftz_oseen}, including the convection trilinear form.
The local-global splitting framework and the construction of the complement spaces $\IL(T)$ are developed in \cref{sec:local_global}.
Based on this, \cref{sec:reduced_formulation} introduces the reduced formulation in which the higher-order pressure components are locally tied to the velocity unknown.
The core stability and error analysis of the Oseen Trefftz-DG method is carried out in \cref{sec:oseenanalysis}.
Numerical results, including Oseen--Kovasznay tests and parameter studies, are presented in \cref{sec:numerics}.
Finally, auxiliary discrete estimates and concrete choices of the convection form are collected in \cref{sec:appendix,sec:choices_ch}.

\subsection{The Oseen problem}\label{sec:model}
The Oseen problem is the canonical linearization of the steady incompressible Navier--Stokes equations around a given convection field.
It is obtained by freezing the advecting velocity in the nonlinear term, replacing $(u\cdot\nabla)u$ by $(w\cdot\nabla)u$ for a prescribed field $w$.
As such, the Oseen problem serves as the fundamental linear building block for Navier--Stokes theory and, in particular, for fixed-point and iterative linearization arguments.

Consider an open bounded Lipschitz domain $\Omega \subset \mathbb{R}^d$ with $d=2, 3$.
The central model problem of this paper is the \emph{Oseen problem}, obtained by linearizing the steady Navier--Stokes equations around a given convection field $w \in [L^{4}(\Omega)]^d$:
Find velocity $u: \Omega \to \mathbb{R}^d$ and pressure $p: \Omega \to \mathbb{R}$ such that
\begin{eqs}\label{eq:oseen}
-\nu\Delta u + (w\cdot\nabla)u + \nabla p &= f, \quad \text{in } \Omega, \\
\Div u &= 0, \quad \text{in } \Omega, \\
u &= 0, \quad \text{on } \partial\Omega,
\end{eqs}
where $f: \Omega \to \mathbb{R}^d$ is the external body force and $\nu > 0$ is the dynamic viscosity.
For the ease of presentation we only consider homogeneous Dirichlet boundary conditions.
The Navier--Stokes problem (obtained by setting $w=u$ in~\eqref{eq:oseen}) is treated in Part~II, \cite{SVLL2_ARXIV_2026}.

\section{Embedded Trefftz-DG discretization}\label{sec:trefftz_oseen}

\subsection{DG discretization of the Oseen problem}\label{sec:oseen}

We consider a sequence of shape-regular simplicial meshes $\Th$ of a polygonal domain $\dom$.
We introduce the underlying DG space given by
\begin{equation} \tag{$\IX$}
\IX^k(\Th) = [\IP^k]^d(\Th) \times \IP^{k-1}(\Th) /  \IR
\end{equation}
where $\IP^k(\Th)$ denotes the broken space of polynomials of degree at most $k$ on each element $T \in \Th$.
We adopt three notational conventions throughout this paper to keep the number of indices manageable (they will be plentiful enough regardless):
\begin{enumerate}[leftmargin=2em,itemsep=0pt,topsep=2pt]
  \item We drop the customary subscript $h$ from discrete spaces, even though all spaces below are mesh-dependent. This departs from much of the FEM and DG literature.
  \item All discrete spaces are element-wise defined on $\Th$; when no mesh argument is written, the full mesh $\Th$ is understood. Thus $\IX^k$ stands for $\IX^k(\Th)$, while $\IX^k(T)$ refers to the local space on a single element $T\in\Th$. The same convention applies to all discrete spaces $\IT_w$, $\IL$, $\IQ$, $\IP^k$, $\IW$, \ldots\ introduced below.
  \item The superscript $\cdot^k$ indicating the polynomial degree is suppressed throughout, unless the specific degree is important in context, for instance $\IX = \IX^k (=\IX^k(\Th))$. For the plain polynomial spaces $\IP^k$, $\IP^{k-1}$, $\IP^{k-2}$ the superscript is always retained.
\end{enumerate}
Let $\Fh=\Fhi\cup\Fhb$ be the set of all (interior and boundary, resp.) facets.
On $F\in\Fhi$ we use the usual definitions of averages $\avg{\cdot}$ and jumps $\jmp{\cdot}$ across facets.
On $F\in\Fhb$ we set $\avg{\phi}=\phi$ and $\jmp{\phi}=\phi$.
With abuse of notation, we write $h$ both for the global mesh size $h:=\max_{T\in\Th}h_T$ and for the corresponding piecewise constant mesh-size functions on $\Th$, $\Fh$, and $\partial\Th$.
Thus $h|_T=h_T$ on elements, $h|_F=h_F:=\diam(F)$ on facets, and on element-boundaries $F\subset\partial T$ we may take $h|_F=h_T$. 
By shape regularity, these choices are uniformly equivalent, $h_F\simeq h_T$ for $F\subset\partial T$.

We define the standard interior penalty DG bilinear forms for the diffusion and pressure-velocity coupling as follows:
\begin{subequations}
    \begin{align} \label{eq:sipdg}
        \blfvis(u_h,v_h)\! &\coloneqq (\nu\! \Grad\! u_h,\Grad\! v_h)_\Th\! -\! (\avg{ \nu \partial_n u_h },\jmp{v_h})_\Fh \!-\! (\avg{\nu \partial_n v_h},\jmp{u_h})_\Fh  \!+\!  \frac{\lambda \nu }{h} (\jmp{u_h},\jmp{v_h})_\Fh, \\ 
        \blfdiv(v_h,p_h)\! &\coloneqq -(\Div v_h, p_h)_{\Th} \!+\! (\jmp{v_h \cdot n}, \avg{p_h})_{\Fh},
        \label{eq:divdg}
    \end{align}
\end{subequations}
where we used the notation $(\cdot,\cdot)_\Th := \sum_{T\in\Th}(\cdot,\cdot)_T$.
The interior penalty parameter $\lambda = \mathcal{O}(k^2)$ is chosen sufficiently large (depending on $k$ and shape regularity) and we used the notation $\partial_n w \coloneqq \nabla w \cdot n$ where $n$ is the unit normal to the facet.
We are ready to introduce the bilinear form $\blfoseen{w}$, writing $\up_h:=(u_h,p_h)$ and $\vq_h:=(v_h,q_h)$, as
\begin{align}\tag{$\blfoseen{w}$}
    \blfoseen{w}(\up_h,\vq_h) := \blfvis(u_h,v_h) + \blfdiv(u_h,q_h) + \blfdiv(v_h,p_h) + \blfconv{w}(u_h,v_h),
\end{align}
where $\blfconv{w}$ is an abstract trilinear form discretizing the convective term $(w\cdot\nabla)u$; its precise assumptions and concrete choices are discussed below in \cref{sec:ch}.
The corresponding DG method then reads as: Find $\up_h \in \IX$, s.t. 
\begin{equation} \tag{DG}\label{eq:DG}
 \blfoseen{w}(\up_h,\vq_h) = (f, v_h)_\Th \quad \forall \vq_h \in \IX.
\end{equation}

\subsection{Local operator and Trefftz condition}

We now introduce our Trefftz space as the local kernel of a relaxed PDE operator.
Rather than requiring functions to lie in the kernel of the strong PDE operator---which may be trivial for $w\ne 0$---we relax the condition by testing only against a suitable lower-order polynomial space.
We therefore first introduce the test space
\begin{equation}\tag{$\IQ$}
    \IQ(T) = [\IP^{k-2}(T)]^d \times \IP^{k-1}(T), \quad \IQ=\IQ(\Th) := \prod_{T\in\Th} \IQ(T),
\end{equation}
where we write elements as $\rs_h = (r_h,s_h) \in \IQ(T)$.

With $\IXreg(T) := [H^2(T)]^d\times H^1(T)$, we define the local PDE operator $\oposeenelem{w}: \IXreg(T) \to \IQ(T)'$ directly via the duality pairing
for $\rs_h = (r_h,s_h) \in \IQ(T)$ and 
$\up = (u,p) \in \IXreg(T)$:
\begin{equation}\tag{$\oposeenelem{w}$}\label{eq:BKw}
    \langle \oposeenelem{w} \up,\, \rs_h \rangle_T
    :=
    h\nu^{-\frac12}\big(-\nu\Delta u + w\cdot\nabla u + \nabla p,\,r_h\big)_T
    +
    \nu^{\frac12} \big(-\Div u,\,s_h\big)_T.
\end{equation}
Accordingly, we set $\oposeen{w}: \IXreg \to \IQ'$ via $\langle \oposeen{w} \up_h,\, \rs_h \rangle_\Th = \sum_{T\in\Th} \langle \oposeenelem{w} \up_h,\, \rs_h \rangle_T$. Here $\IXreg = \IXreg(\Th) = [H^2(\Th)]^d \times H^1(\Th)$.
The scaling factor $h\nu^{-\frac12}$ in the first component balances the scaling with respect to our choice of norms, as we will see in the analysis.

With this PDE operator defined, we set the Trefftz space as the subspace of $\IX$ satisfying the Trefftz condition element-wise:
\begin{align}\label{eq:Trefftzconditions} \tag{$\IT_{w}$}
    \IT_{w} :=&
   \{ \up_h \in \IX \mid 
    \langle \oposeen{w} \up_h,\rs_h\rangle_T = 0
    \quad\forall \rs_h\in \IQ(\Th)\}.
\end{align}
For $w=0$ (the Stokes case), 
the residuals $-\nu\Delta u_h+\nabla p_h$ and $-\Div u_h$ already lie in $\IQ(T)$ for any $\up_h\in\IX(T)$, $T\in\Th$, so $\oposeen{0}$ coincides with the (element-wise) strong Stokes operator.
For $w\neq 0$, the pairing $(w\cdot\nabla u_h,r_h)_\Th$ implicitly involves only the $L^2$-projection of $w\cdot\nabla u_h$ into $[\IP^{k-2}]^d$, i.e.\ functions in $\IT_w$ fulfill the PDE only in a relaxed sense.

\subsection{The embedded Trefftz-DG formulation}

With the above definitions, we are ready to introduce the embedded Trefftz-DG method for the Oseen problem.
The Trefftz-DG method to \eqref{eq:oseen} then reads: Find $\up_h = (u_h,p_h)\in \IX$ so that
\begin{subequations} \label{eq:tdgoseen}
  \begin{align}
\blfoseen{w}(\up_h,\vq_h) &= (f,v_h)_{\Th}  && \forall \vq_h = (v_h,q_h) \in \IT_{w} \label{eq:tdgoseen:1}\\ 
\big\langle \oposeen{w}\up_h,\rs_h\big\rangle_\Th &= \big(h \nu^{-\frac12}f,r_h\big)_\Th && \forall \rs_h = (r_h,s_h) \in \IQ. \label{eq:tdgoseen:2}
  \end{align}
\end{subequations}
\cref{eq:tdgoseen} is a coupled system for the unknown $\up_h$ in the ambient DG space $\IX$, where the first line corresponds to the (global) problem on the Trefftz space and the second line corresponds to \emph{local problems} on each element.
For compact notation of the coupled system \eqref{eq:tdgoseen}, we introduce the product test space $\IZ_{w}:=\IT_w\times \IQ$ and 
introduce the bilinear and linear forms
\begin{align}\tag{$\blftrefftz{w}$}
    \blftrefftz{w}(\up_h,\z_h) &:= 
    \blfoseen{w}(\up_h,\vq_h) +
    \langle \oposeenelem{w} \up_h,\rs_h\rangle_\Th,
    & \up_h &\in \IX,\\
\tag{$F$}
F(\z_h) &:= (f,v_h)_\Th + (h\nu^{-\frac12}f,r_h)_\Th,
& \z_h = (\vq_h,\rs_h) = ((v_h, q_h), (r_h, s_h)) &\in \IZ_w.
\end{align}
so that \eqref{eq:tdgoseen} also reads: Find $\up_h \in \IX$, so that $\blftrefftz{w}(\up_h,\z_h) = F(\z_h)$ for all $\z_h \in \IZ_w$.

\begin{figure}[htbp!]
 \centering
  \vspace*{-0.2cm}
\begin{tabular}{ccc}
 \includegraphics[width=0.45\textwidth,page=1]{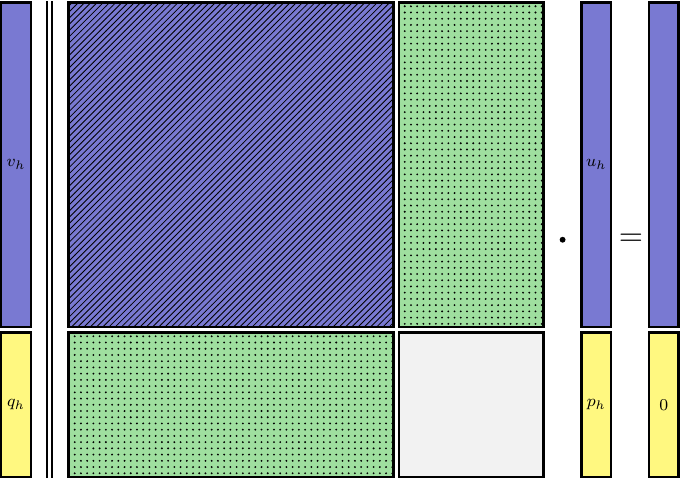}
 &~~&
 \includegraphics[width=0.45\textwidth,page=3]{matrix_schematic.pdf} 
\end{tabular}
  \vspace*{-0.3cm}
\caption{Schematic comparison between Standard DG formulation in velocity and pressure unknowns and equations (left) vs. Trefftz formulation as a block triangular system coupling Trefftz solution and local solution (right).}
 \label{fig:block_structure_comparison}
 \end{figure}

This can be interpreted as a (upper) block triangular system where the lower right block is additionally block-diagonal with blocks given by the local operator $\oposeen{w}$. This hence yields a different algebraic structure as the usual DG method, cf.\ \cref{fig:block_structure_comparison} for a schematic comparison of the algebraic structure of the standard DG method and the embedded Trefftz-DG method.

\paragraph{On the implementation.}
The system \eqref{eq:tdgoseen} is a coupled upper-triangular block system that can be solved by first solving the local problems on each element, possibly in parallel, cf.\ \cref{fig:block_structure_comparison},
which allows us to compute the Trefftz space via embedding in the same breath.
And then solving the global problem on the much smaller Trefftz space.
In this sense, the solution is split into a local part $\up_\IL = (u_\IL,p_\IL)$ and a global part $\up_\IT = (u_\IT,p_\IT) \in \IT_{w}$, where the local part is given by the solution of the local problems and the global part is then solved for by
\[
    \blfoseen{w}(\up_\IT,\vq_\IT) = (f,v_\IT)_{\Th} - 
    \blfoseen{w}(\up_\IL,\vq_\IT)  \quad \forall \vq_\IT = (v_\IT,q_\IT) \in \IT_{w}.
\]
The local problems, can be solved on each element by computing a pseudo-inverse of the local operator $\oposeenelem{w}: \IX(T) \to \IQ(T)'$.
Using an QR-decomposition or SVD of the matrix representation of $\oposeenelem{w}$, we can compute simultaneously an embedding for the Trefftz space $\IT_{w}$ into the ambient $\IX$ space and the pseudo-inverse of $\oposeenelem{w}$.
The implementation is discussed in more detail in \cite{LS_IJMNE_2023}.

We note that the Trefftz-DG method presented here yields a solution in $\IX$. Above, we defined the Trefftz subspace $\IT_w$, but not a complement space to $\IX$. We do not need to explicitly construct such a complementary space for the realization of the method. In contrast, for the analysis we will consider a specific complementary space $\IL$ so that $\IX = \IT_w \oplus \IL$.

\paragraph{Norms and projections.}
Below, we use the mesh dependent DG norms for $u\in [H^1(\Th)]^d$, $v\in [H^2(\Th)]^d$:
\begin{align} \tag{$\Vert \cdot \Vert_{1,h(,\ast)}$}
  \norm{u}_{1,h}^2 \!:= \!\norm{\Grad u}_{\Th}^2 \!+\! \norm{ h^{-\frac12} \jmp{u}}^2_{\Fh},~~~
    \norm{v}_{1,h,\ast}^2 \!:=\! \norm{v}_{1,h}^2 \!+\! \norm{ h^{\frac12} \ \partial_n v}^2_{\partial \Th}
    + \norm{h \Delta v}^2_{\Th},
\end{align}
For $u_h\in[\IP^k]^d$, the inverse inequality \eqref{la}, and a trace inequality shows that the two versions of the mesh-dependent norm are equivalent on $[\IP^k]^d$.

Let us also equip the spaces $\IX$ and $\IZ_w$ with suitable norms.
For $(u,p) \in \IXreg$ and $\z_h=(\vq_h,\rs_h) \in \IZ_w$ we introduce the following mesh-dependent norms:
\begin{align*}
  \norm{p}_{0,h}^2&:= \norm{ h \Grad p}_{\Th}^2 + \norm{ h^\frac12 \jmp{\Pi^0 p}}_{\Fhi}^2,  
  & 
  \dgnorm{(u,p)}^2 &:= \nu \norm{u}_{1,h}^2 + \nu^{-1} \norm{p}_{0,h}^2, 
  \\
  \dgsnorm{(u,p)}^2 &:= \nu \norm{u}_{1,h,\ast}^2 + \nu^{-1} \norm{p}_{0,h}^2,
  &
  \|\z_h\|_{\IZ}^2 &:=\dgnorm{\vq_h}^2 + \|\rs_h\|_{\IQ}^2,
\end{align*}
where for $\rs_h = (r_h,s_h) \in \IQ$ we introduce the norm
\[
    \|\rs_h\|_{\IQ}^2 :=\|r_h\|_{\Th}^2+\|s_h\|_{\Th}^2. 
\]
We let $\|\cdot\|_{\IQ'}$ be the induced dual norm; by the standard $L^2$-identification we have $\|\cdot\|_{\IQ'}=\|\Pi_{\IQ}\cdot\|_{\IQ}$ where $\Pi_{\IQ}$ is the element-wise orthogonal projection from $[L^2(T)]^d \times L^2(T)$ to $\IQ(T)$.

On the broken space $W^{1,4}(\Th)$ we introduce 
\begin{equation}\label{eq:w14h}\tag{$|\cdot|_{1,4,h}$}
|w|_{1,4,h}
:=
\|w\|_{L^4(\Th)}
+
\|h \nabla w\|_{L^4(\Th)}
+
\|h^{\frac14}\jmp{w\cdot n}\|_{L^4(\Fh)} .
\end{equation}

Throughout, $\Pi^\ell_S$ denotes the $L^2$-orthogonal projection onto $\IP^\ell(S)$ for scalar functions, or onto $[\IP^\ell(S)]^d$ for vector-valued functions (the dimension is clear from context). More generally, $\Pi$ with suitable super- and subscripts always denotes an $L^2$-orthogonal projection; for instance $\Pi^0$ is the elementwise mean, and $\Pi_T^{k-2}$ projects onto $[\IP^{k-2}(T)]^d$.

\subsection{The convection trilinear form}\label{sec:ch}

To discretize the convective term $(w\cdot\nabla)u$ many choices are possible.
For now, we keep the discretization of the convection term abstract, and denote it as trilinear form $\blfconv{w}$.
We give concrete possible (and standard) choices for the trilinear forms in \cref{sec:choices_ch}.
The analysis of the method only relies on some abstract assumptions on the trilinear form which are fulfilled for the standard choices.

We define the subspace of admissible convection fields as
\begin{equation}\label{eq:IWh}\tag{$\IW$}
    \IW := [W^{1,4}(\Th)]^d.
\end{equation}
We define the resolution quantity
\begin{equation}\tag{$|\cdot|_{h,d}$}
|w|_{h,d}
:=
\nu^{-1}\Big(
\max_{T\in \Th} C_4 \|h^{1-\frac{d}{4}} w\|_{L^4(T)}
+
\max_{T\in \Th} C_{W,4}\|h^{2-\frac{d}{4}} \nabla w\|_{L^4(T)}
+
\max_{F\in \Fh} C_{J,4}\|h^{\frac14}\jmp{w\cdot n}\|_{L^4(F)}
\Big),
\end{equation}
where $C_4$ is the constant from the local discrete embedding \eqref{l4} and $C_{W,4}$ is the facet trace constant from \eqref{eq:w14_trace}.
Notice that for fixed $w$, the quantity $|w|_{h,d}$ goes to zero as $h\to 0$; this is the basis for a resolution condition.

For polynomial convection fields $w_h\in[\IP^k]^d$, the quantity $|w_h|_{h,d}$ is weaker than the discrete $H^1$ norm by a factor $h^{1-\frac d4}$:
\begin{equation}\label{eq:gamma_norm_bound}
|w_h|_{h,d}
\leq \nu^{-1} C_\gamma h^{1-\frac{d}{4}} \| w_h\|_{1,h}.
\end{equation}
Indeed, the first term is controlled by the discrete Sobolev embedding \eqref{se}, the second and third by inverse estimates and the polynomial trace inequality \eqref{eq:L4_trace_poly}.

\begin{assumption}[\emph{Conditions on $\blfconv{w}$}]\label{ass:ch}
There exists $\contconv>0$ (depending only on shape-regularity and $k$) such that for all $w\in \IW$, $u\in [H^2(\Th)]^d$, and $v_h\in[\IP^k]^d$,
\begin{equation}\label{eq:ch_cont}
  |\blfconv{w}(u,v_h)|
  \le
  \contconv\,|w|_{1,4,h}\,\|u\|_{1,h,\ast}\,\|v_h\|_{1,h}.
\end{equation}
We additionally assume
\begin{equation}\label{eq:ch_nonneg}
\blfconv{w}(u_h,u_h) \ge 0\qquad\forall w\in\IW,\ \forall u_h\in[\IP^k]^d,
\end{equation}
and the mixed $L^4$-estimate
\begin{equation}\label{eq:ch_mixed_L4}
|\blfconv{w}(z,v_h)|
\le C_{t,\ast}\nu|w|_{h,d}\|z\|_{1,h,\ast}\|v_h\|_{1,h},
\qquad \forall w\in\IW,\ \forall v_h\in[\IP^k]^d,\ \forall z\in \IX_{\rm ps}(\Th),
\end{equation}
where $\IX_{\rm ps}\subset [H^2(\Th)]^d$ is a space that satisfies a Poincar\'e estimate, i.e.
on all $K\in\Th$ and for all $z$ in $\IX_{\rm ps}$ there holds
$\|z\|_K\leq C h_K \|\nabla z\|_K$.
\end{assumption}
In the analysis below, $\IX_{\rm ps}(\Th)$ will be instantiated either by the local complement space $\ILu$ or by approximation-error classes $u-\vr_h$ satisfying the elementwise mean constraint $\Pi^0\vr_h=\Pi^0 u$.
Concrete choices of $\blfconv{w}$ satisfying \cref{ass:ch} are given in \cref{sec:choices_ch}.

In this work, we will further make the \emph{resolution assumption} 
\begin{equation}\label{ass:wh}\tag{A1}
    w \in \IW \quad\text{and}\quad |w|_{h,d}\leq c_\gamma,
\end{equation}
for some fixed $c_\gamma$, depending only on $k$ and the shape-regularity of the mesh elements.
Note that this is not a smallness assumption on the data, but rather a resolution assumption, as $|w|_{h,d} \to 0$ for $h \to 0$. This resolution assumption will play a crucial role in the analysis of the validity of the Trefftz method, and the analysis of the Oseen problem.

\section{Local-global splitting}\label{sec:local_global}

Throughout, $C>0$ denotes a generic constant depending only on $k$ (and $d$), and may change from line to line.

\subsection{Construction of \texorpdfstring{$\IL$}{IL}}\label{ssec:IL}

While an explicit construction of the complement space to the Trefftz space is not needed for the implementation of the method, 
a crucial aspect of the \emph{analysis} is the choice of a proper local spaces $\IL(T)$ that complement $\IX(T)$, so that $\IX(T) = \IL(T) \oplus \IT_{w}(T)$ for all $T \in \Th$ and hence $\IX(\Th) = \IL(\Th) \oplus \IT_{w}(\Th)$. The specific choice below enables the stability and well-posedness analysis of the coupled system.
    
We define the local complement space component-by-component, i.e. we define the velocity and pressure components separately, and then combine them to get $\IL(T) = \ILu(T) \times \ILp(T)$. Correspondingly, we have $\dim(\IL(T)) = \dim(\ILu(T)) + \dim(\ILp(T))$ which needs to match $\dim(\IQ(T)) = \dim(\IX(T)) - \dim(\IT_{w_h}^k(T)) = d \cdot \dim(\IP^{k-2}(T)) + \dim(\IP^{k-1}(T))$.

\subsubsection{Velocity component of the complement space}\label{sssec:Lu}

The construction of the velocity component $\ILu(T)$ of $\IL(T)$ is done in three steps. On each element $T \in \Th$:
\begin{enumerate}
    \item We first consider polynomials of degree $k$ that vanish on a ball $B_T$ associated to $T$.
    \item We further enrich by $\Span\{\chi, n_{T,1}, \ldots, n_{T,d}\}$ to get control over the facet mean fluxes
    \item Finally, we constrain the mean flux on $d$ (of $d+1$) facets to be zero.
\end{enumerate}

For each element $T \in \Th$, we denote by $n_{T,1},\ldots,n_{T,d+1}$ the unit outward normals to the facets of $T$ and by $F_{T,1},\ldots,F_{T,d+1}$ the corresponding facets. We also denote by $\chi = x - x_T$ the local coordinate vector centered at the barycenter $x_T$ of element $T$ and define $B_T$ as the largest ball \emph{centered at $x_T$} contained in $T$, the radius of which we denote as $h_T$, cf.\ \cref{fig:IL_geometry}. This gives the precise meaning of the mesh-size parameter $h_T$ used throughout the paper; it is equivalent to the element diameter up to a constant depending only on the shape-regularity of $\Th$.
\begin{figure}[htbp!]
\vspace*{-0.8cm}
\centering
\begin{subfigure}[b]{0.35\textwidth}
\centering
\begin{tikzpicture}[scale=1.9]
  \coordinate (A) at (0, 0);
  \coordinate (B) at (2, 0);
  \coordinate (C) at (0.7, 1.5);
  
  \coordinate (G) at ($(A)!0.5!(B)!0.33!(C)$);
  
  \draw[] (A) -- (B) -- (C); %
  \draw[gray!50!purple, thick] (C) -- (A); %
  
  \filldraw[black] (G) circle (1pt) node[anchor=south west, inner sep=3pt] {$x_T$};
  
  \draw[dashed, red,fill=teal!20] (G) circle (0.49cm);
  \draw[] (G) circle (0.02cm) node[anchor=north east, scale=0.7] {$x_T$};
  
  \draw[->] (G) -- ($(G)+(0.2,0)$) node[scale=0.7, anchor=north] {$\chi_1$};
  \draw[->] (G) -- ($(G)+(0,0.2)$) node[scale=0.7, anchor=east] {$\chi_2$};

  \node[teal!50!black, anchor=south] at ($(G)+(-0.54,-0.45)$) {$B_T$};
  
  \node[anchor=north east] at ($(A)!0.5!(B)$) {$F_{T,1}$};
  \node[anchor=west] at ($(B)!0.4!(C)$) {$F_{T,2}$};
  \node[anchor=east] at ($(C)!0.6!(A)$) {$F_{T,3}$};

  \node[anchor=west] at (B) {$T$};

  \coordinate (mid1) at ($(A)!0.5!(B)$);
  \coordinate (mid2) at ($(B)!0.5!(C)$);
  \coordinate (mid3) at ($(C)!0.5!(A)$);
  
  \draw[-{Latex[length=2mm]}, thick] (mid1) -- ++(0, -0.35) node[anchor=north, inner sep=2pt] {$n_{T,1}$};
  \draw[-{Latex[length=2mm]}, thick] (mid2) -- ++(0.3, 0.25) node[anchor=west, inner sep=2pt] {$n_{T,2}$};
  \draw[-{Latex[length=2mm]}, thick] (mid3) -- ++(-0.3, 0.15) node[anchor=east, inner sep=2pt] {$n_{T,3}$};
\end{tikzpicture}
\caption{Geometry of element $T$. \\ \vphantom{argh}}
\label{fig:IL_geometry_tikz}
\end{subfigure}%
\hfill
\begin{subfigure}[b]{0.6\textwidth}
\centering
\includegraphics[width=0.8\textwidth]{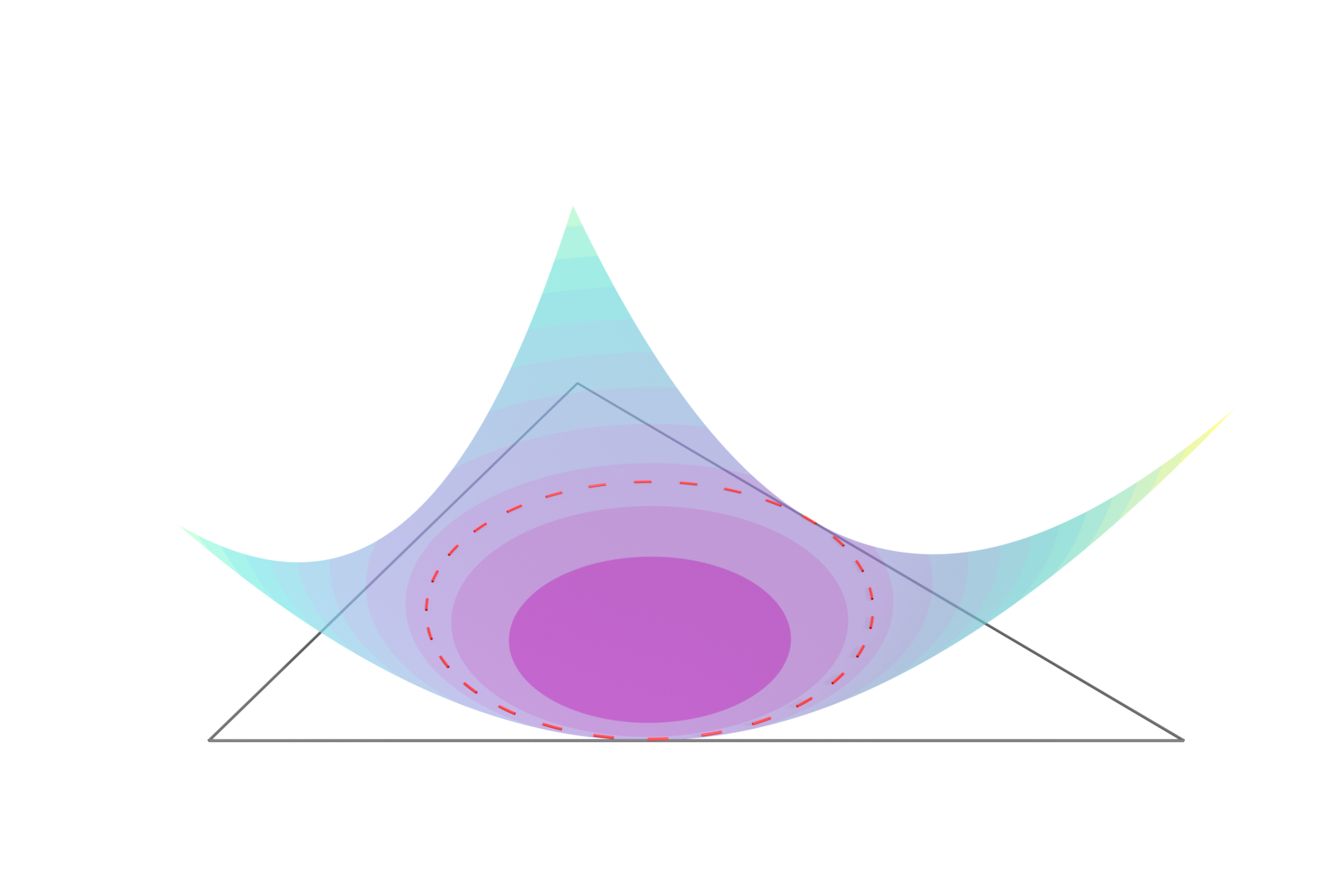}
\vspace*{-0.2cm}
\caption{Elevation plot of a bubble function $(|\chi|^2-h_T^2)\cdot q$ on $T$; the ball boundary $\partial B_T$ is highlighted in (dashed) red.}
\label{fig:IL_bubble}
\end{subfigure}
\caption{Sketch of element geometry and an example bubble function as in $\ILu(T)$.}
\label{fig:IL_geometry}
\end{figure}
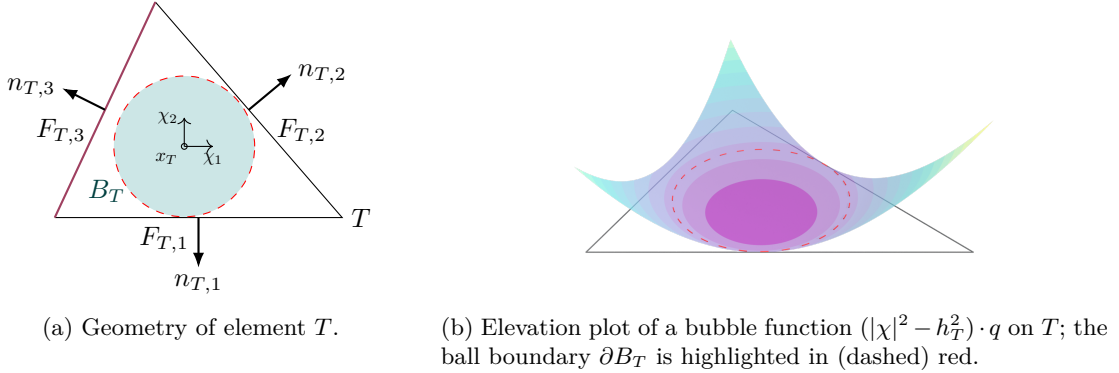

We define the space of vector-valued polynomials of degree at most $k$ on $T$ that vanish on $\partial B_T$ as 
\[
    \IB^k_0(T):= \IP^k(T) \cap H^1_0(B_T).
\]
This is the first part in the definition of each velocity component of $\ILu(T)$. The construction of this space allows to relate the Laplacian appearing in $\oposeenelem{w}$ to the $H^1$-semi norm --- at least on $B_T$ --- through partial integration: $(-\Delta u_h, q_h)_{B_T} = (\nabla u_h, \nabla q_h)_{B_T}$ for all $u_h \in [\IB^k_0(T)]^d$ (and suitable $q_h$). A ball as the domain for fixing zeros is chosen so that we can factor out the quadratic bubble $(|\chi|^2-h_T^2)$ and characterize
\[
    \IB^k_0(T) = (|\chi|^2-h_T^2)\IP^{k-2}(T) \quad \Rightarrow \quad \dim([\IB^k_0(T)]^d) = d \cdot \dim(\IP^{k-2}(T)).
\]
The next part in the construction of the velocity component $\ILu(T)$ of $\IL(T)$ is designed to control the net fluxes on the element and the facets. Each element in $\Span\{\chi\}$ has constant divergence and is introduced to have an explicit degree of freedom --- lying in the kernel of $-\nu \Delta$ --- to control the mean flux on the element. The additional $d$-dimensional space of constant vectors $\Span\{ n_{T,1}, \ldots, n_{T,d}\} (= \mathbb{R}^d)$ is introduced to similarly control the mean flux on the first $d$ facets of the element while lying in the kernel of $\oposeenelem{w}$. 

In the final step, we make use of the added degrees of freedom in the second step to constrain the mean flux on the first $d$ facets to zero. We hence define:
\begin{equation}\tag{$\ILu$}
    \ILu(T) := \{ u_h \in [\IB^k_0(T)]^d + \Span\{\chi, n_{T,1}, \ldots, n_{T,d}\} : \int_{F_{T,i}} u_h \cdot n_{T,i} = 0, i=1,\ldots,d\},
\end{equation}
and find $\dim(\ILu(T)) = \dim([\IB^k_0(T)]^d) + d + 1 - d = \dim([\IP^{k-2}(T)]^d) + 1$. Note that the remaining total facet flux on $F_{T,d+1}$ compensates for the accumulated divergence on the element in the sense that 
\begin{equation} \label{eq:flux_divergence_relation}
\int_T \Div u_h = \sum_{i=1}^{d+1} \int_{F_{T,i}} u_h \cdot n_{T,i} = \int_{F_{T,d+1}} u_h \cdot n_{T,d+1} \qquad \forall u_h \in \ILu(T),
\end{equation}

The main motivation for the introduction of the constants $\Span\{n_{T,1},\dots,n_{T,d}\}$ alongside the constraints flux constraints on the edge $F_{T,i}$, $i=1,..,d$ is the following simple observation due to \eqref{eq:flux_divergence_relation}: 
\begin{lemma}\label{eq:divrelation}
For all $u_h \in \ILu$ with $\Div u_h|_\Th = 0$ we have $\Pi_F^0 (u_h \cdot n)= 0$ on all facets $F \in \Fh$. 
\end{lemma}

\begin{remark}\label{rem:ILu_poincare}
The space $\ILu(T)$ contains no nonzero constant vectors, i.e.
$ \ILu(T)\cap [\IP^0(T)]^d=\{0\}.$
Consequently, $\ILu(T)$ satisfies a Poincar\'e estimate and in particular, it is an admissible choice for $\IX_{\rm ps}(\Th)$ in \eqref{eq:ch_mixed_L4}.
\end{remark}

\subsubsection{Pressure component of complement space} 
The pressure of the local problem appears only in terms of its gradient. Hence, we choose zero-mean polynomials of degree at most $k-1$ for the pressure component of $\IL(T)$, i.e. we define:
\begin{equation}\tag{$\ILp$}
    \ILp(T) = \{p \in \IP^{k-1}(T) : \int_T p = 0\}.
\end{equation}

\subsubsection{The velocity-pressure complement space} 
Putting both together we define 
\begin{equation}\tag{$\IL$}
    \IL(T) = \ILu(T) \times \ILp(T).
\end{equation}
and find $\dim(\IL(T)) = \dim(\IP^{k-2}(T)) + 1 + \dim(\IP^{k-1}(T)) - 1 = d \cdot \dim(\IP^{k-2}(T)) + \dim(\IP^{k-1}(T)) = \dim(\IQ(T))$, which matches the dimension of the test space for the local problem, as required for the invertibility of the local operator $\oposeenelem{w}$ on $\IL(T)$. We then compose canonically, $\IL = \IL(\Th) = \prod_{T\in\Th} \IL(T)$.

\subsection{Local operator}\label{ssec:local_op}
Having defined the complement space $\IL(T)$ in \cref{ssec:IL}, we now establish that the local operator $\oposeenelem{w}:\IL(T)\to\IQ'(T)$ is boundedly invertible.
This is the key structural result that justifies the local-global splitting: invertibility of $\oposeenelem{w}$ on $\IL(T)$ implies the direct-sum decomposition $\IX = \IT_w \oplus \IL$ (cf.\ \cref{cor:local_invertibility}), so every discrete function in $\IX$ can be uniquely split into a Trefftz component in $\IT_w$ and a complement component in $\IL$, and within the Trefftz DG method the latter can be solved for (locally on each element independently).

The analysis is carried out in two steps.  We first treat the \emph{Stokes case} $w=0$ in \cref{sssec:Stokes}. 
We then extend to the \emph{Oseen case} \cref{sssec:Oseen} $w\neq 0$ by treating the convection term as a perturbation and showing that invertibility is preserved whenever  resolution is sufficient.

\subsubsection{Stokes case \texorpdfstring{$w=0$}{w=0}}\label{sssec:Stokes}
We start off with the following continuity result for the local Stokes operator $\oposeenelem{0}$.
\begin{lemma}[Continuity of $\oposeenelem{0}$]\label{lem:AKw_continuity}
For all $\up = (u,p)\in \IXreg$ there holds
\begin{equation}\label{eq:AKzero_continuity}
    \|\oposeen{0} \up \|_{\IQ'}^2 = \sum_{T\in\Th}\|\oposeenelem{0} \up \|_{\IQ(T)'}^2
    \le
    \contlocal{0}^2 \,\,\dgsnorm{ \up }^2,
\end{equation}
with $\contlocal{0}>0$ depending only on the shape-regularity of $\Th$ and on $k$.
\end{lemma}
\begin{proof}
The $L^2$ representation of $\oposeen{0} \up$ in $\IQ$ is $(h_T\nu^{-\frac12} \Pi^{k-2} (-\nu \Delta u + \nabla p), - \Pi^{k-1} \Div u )$.   
With the continuity of the $L^2$ projection we then have on each element $T\in\Th$
\[
\|\oposeenelem{0}\up\|_{\IQ(T)'}^2
\le
\|h_T\nu^{-\frac12} (-\nu\Delta u + \nabla p) \|_T^2
+\nu \|\Div u\|_T^2,
\]
Using Cauchy-Schwarz inequality 
 and $\|\Div u_h\|_T\le \|\nabla u_h\|_T$ we bound 
\begin{align*}
\|\oposeenelem{0}\up_h\|_{\IQ(T)'}^2
\le C\Big(
    \nu\,\| \nabla u\|_T^2 + \nu\,\| h_T \Delta u\|_T^2
    + \nu^{-1}\|h_T\nabla p_h\|_T^2
\Big).
\end{align*}
Summing over $T\in\Th$ gives \eqref{eq:AKzero_continuity}.
\end{proof}

We now wish to verify that the local Stokes operator $\oposeen{0}$ is invertible on a suitable subspace of $\IX$.  As a intermediate result, we use the following estimate for functions in the bubble space $[\IB^k_0(T)]^d = [\IP^k(T) \cap H^1_0(B_T)]^d$ enriched by $\Span\{\chi\}$. 
\begin{lemma}[Auxiliary estimate for a bubble space]\label{lem:ball}
    There exists a constant $C>0$ depending only on $k$ (and $d$) such that for all
    $u_h\in [\IB^k_0(T)]^d \oplus \Span\{\chi\}$ and all $p_h\in \IP^{k-1}(T)$ there holds
    \begin{equation}\label{eq:ball_est}
        \|\nabla u_h\|_{T}^2
        \le
        C\Big(
            \frac{h_T^2}{\nu^2}\|-\nu \Delta u_h + \nabla p_h\|_{T}^2
            + \|\Div u_h\|_{T}^2
        \Big).
    \end{equation}
\end{lemma}
\begin{proof}
We will show \eqref{eq:ball_est} w.r.t. to the $L^2(B_T)$-norms, which is sufficient since  the norms on $T$ and $B_T$ are equivalent on the finite-dimensional space $[\IB^k_0(T)]^d \oplus \Span\{\chi\}$.
In the following we will derive a splitting $u_h = u_\ell + u_0 = u_\ell + u_{\perp} + u_{\Div}$ for which the obvious estimate 
\[
    \|\nabla u_h\|_{B_T}^2
    \le C\big(\|\nabla u_0\|_{B_T}^2+\|\nabla u_\ell\|_{B_T}^2\big)
    = C\big(\|\nabla u_{\Div}\|_{B_T}^2+\|\nabla u_\perp\|_{B_T}^2+\|\nabla u_\ell\|_{B_T}^2\big)
\]
holds. Subsequently, we will derive the splitting and a bound for each of these summands. 

\noindent
\emph{1. Split off the affine part.}
Set
\[
    u_\ell := \frac{\Pi^0 (\Div u_h)}{\Div (\chi)} \,\chi = d^{-1} \Pi^0 (\Div u_h) \, \chi \in \Span\{\chi\},
    \qquad
    u_0:=u_h-u_\ell\in \IB^k_0,
\]
so that $\Pi^0 \Div u_h = \Pi^0 \Div u_\ell$ and $\Pi^0 \Div u_0 = 0$ and with $\nabla \chi = I$ we find
$\|\nabla u_\ell\|_{B_T}\le d^{-\frac12} \|\Div u_h\|_{B_T}$.

\noindent
\emph{2. $H^1_0$-orthogonal decomposition of $u_0$ into $u_{\Div}$ and $u_\perp$.} %
Let
\[
    \IB_{\Div}(T):=\{v\in \IB^k_0(T) \mid \Div v=0\},
\]
and let $\IB_{\Div}^\perp(T)$ be its orthogonal complement in $\IB^k_0(T)$ with respect to
$(\nabla\cdot,\nabla\cdot)_{B_T}$.
Then $u_0=u_{\Div}+u_\perp$ with $u_{\Div}\in \IB_{\Div}(T)$, $u_\perp\in \IB_{\Div}^\perp(T)$, and
\[
    \|\nabla u_0\|_{B_T}^2=\|\nabla u_{\Div}\|_{B_T}^2+\|\nabla u_\perp\|_{B_T}^2,
    \qquad
    \Div u_0=\Div u_\perp.
\]
Since $v\mapsto \|\Div v\|_{B_T}$ is a norm on the finite-dimensional space $\IB_{\Div}^\perp(T)$, there exists
$c_\perp>0$ depending only on $k$ such that
\begin{align*}
    \|\nabla u_\perp\|_{B_T} &\le c_\perp \|\Div u_\perp\|_{B_T} =  c_\perp \|\Div u_0\|_{B_T}
    \le \|\Div u_h\|_{B_T}+\|\Div u_\ell\|_{B_T} \\ 
    & \le \|\Div u_h\|_{B_T}+\|\nabla u_\ell\|_{B_T} 
    \le (1+d^{-\frac12}) \|\Div u_h\|_{B_T}.
\end{align*}

\noindent
\emph{3. Stokes estimate on the divergence-free part.}
Since $u_{\Div}\in [H^1_0(B_T)]^d$ and $\Div u_{\Div}=0$, we have $(\nabla p_h,u_{\Div})_{B_T}=0$.
Moreover $u_\ell$ is affine, hence $\Delta u_h=\Delta u_0$.
Thus, integrating by parts, Cauchy--Schwarz and Poincar\'e on $B_T$
\begin{align*}
    \nu\|\nabla u_{\Div}\|_{B_T}^2
    & =(\nu\nabla (u_0-  u_\perp),\nabla u_{\Div})_{B_T}
    =(\nu\nabla u_0,\nabla u_{\Div})_{B_T}
    =(-\nu\Delta u_0,u_{\Div})_{B_T} \\
    & =(-\nu\Delta u_h+\nabla p_h,u_{\Div})_{B_T}
    \le \|-\nu\Delta u_h+\nabla p_h\|_{B_T}\,\|u_{\Div}\|_{B_T} \\
    & \le C \|-\nu\Delta u_h+\nabla p_h\|_{B_T}\, h \, \|\nabla u_{\Div}\|_{B_T}.
\end{align*}
Hence,
\[
    \|\nabla u_{\Div}\|_{B_T}^2
    \le C^2\,\frac{h^2}{\nu^2}\,\|-\nu\Delta u_h+\nabla p_h\|_{B_T}^2 .
\]
\end{proof}
With the above estimate at hand, we can now show the following result for the local Stokes operator $\oposeenelem{0}$ on $\IL(T)$.
\begin{theorem}[Stable invertibility of $\oposeen{0}$ on $\IL$]\label{thm:AK0}
For $T \in \Th$, the local Stokes operator $\oposeenelem{0}:\IL(T)\to \IQ(T)'$ is a bijection. Moreover, there exists $\stablocal{0}>0$,
depending only on $k$ and the shape-regularity of $T$ (in particular through the constant in
\Cref{lem:ball}), such that for all $\up_h\in \IL(T)$ there holds
\begin{equation}\label{eq:AK0_coerc}
  \|\oposeenelem{0}\up_h\|_{T} \ge
  \stablocal{0}\Big(
    \nu\|\nabla u_h\|_T^2
    + \nu h_T^{-1}\|u_h\|_{\partial T}^2
    + \nu^{-1}\| h_T\nabla p_h\|_T^2
    \Big)^{\frac12}\!
\end{equation}
Hence, $\oposeen{0}: \IL \to \IQ'$ is also a bijection and $\Vert \oposeen{0} \up_h \Vert_{\IQ'} \ge \stablocal{0} \dgnorm{\up_h}$ for all $\up_h \in \IL$. %
\end{theorem}
\begin{proof}
We set $u_h =\tilde u_h - u_c$ with $\tilde u_h \in [\IB_0^k(T)]^d + \Span\{\chi\}$, $u_c \in [\IP^0(T)]^d = \Span\{n_{T,1},\dots,n_{T,d}\}$.
Explicitly $u_c$ is the unique vector such that $\int_{F_{T,i}} \tilde u_h \cdot n_{T,i} = |F_{T,i}| u_c \cdot n_{T,i}$ for $i=1,\ldots,d$, which is well-defined by the linear independence of $n_{T,1}, \ldots, n_{T,d}$.

The bounds for the three contributions $\nu \Vert\nabla u_h\Vert_T^2$, $\nu h_T^{-1} \Vert u_h \Vert_{\partial T}^2$ and $\nu^{-1} \Vert h_T \nabla p_h \Vert_T^2$ are then shown one after another.

\noindent
\emph{1. Control of $\nu \Vert\nabla u_h\Vert_T^2$.}
Applying \eqref{eq:ball_est} yields
\[
  \nu\|\nabla  u_h\|_T^2
  =\nu\|\nabla \tilde u_h\|_T^2
  \le C\Big(
      \|h_T\nu^{-\frac12}(-\nu\Delta u_h+\nabla p_h)\|_T^2
      + \|\nu^{\frac12}\Div u_h\|_T^2
  \Big)
  = C\,\|\oposeenelem{0}\up_h\|_T^2 .
\]
\emph{2. Control of the boundary term $\nu h_T^{-1} \Vert u_h \Vert_{\partial T}^2$.}
Due to $\int_{F_{T,i}} \tilde u_h \cdot n_{T,i} = |F_{T,i}| u_c \cdot n_{T,i}$ for $i=1,\ldots,d$, we have 
\[
    \|u_h\|_{\partial T} \leq \|u_c\|_{\partial T} + \|\tilde u_h\|_{\partial T}
    \le C \|\tilde u_h\|_{\partial T}. 
\]
It hence suffices to show the bound for $\tilde u_h$. By the trace inequality on $T$ and a Poincar\'e estimate on $[\IB_0^k(T)]^d$ (exploiting that it does not contain constants), 
\begin{align*}
    \nu h^{-1}\|\tilde u_h\|_{\partial T}^2
    \le C \nu \|\nabla \tilde u_h\|_T^2
    \stackrel{\eqref{eq:ball_est}}{\le} C  \Vert \oposeenelem{0} \up_h \Vert_T^2.
\end{align*}
\emph{3. Control of pressure contribution $\nu^{-1} \Vert h_T \nabla p_h \Vert_T^2$.}
By the triangle inequality and the inverse bound $h_T\|\Delta u_h\|_T\le C\|\nabla u_h\|_T$,
\[
    \|h_T\nabla p_h\|_T
    \le \|h_T(-\nu\Delta u_h+\nabla p_h)\|_T + \nu\,\|h_T\Delta u_h\|_T
    \le \|h_T(-\nu\Delta u_h+\nabla p_h)\|_T + C\,\nu\|\nabla u_h\|_T.
\]
Multiplying by $\nu^{-1}$ and using step~1 yields
\[
    \nu^{-1}\|h_T\nabla p_h\|_T^2 \le C\,\|\oposeenelem{0}\up_h\|_T^2.
\]
\noindent
\emph{4. Conclusion and bijectivity.}
Combining the above estimates, we obtain \eqref{eq:AK0_coerc}.
In particular, $\oposeenelem{0}\up_h=0$ implies $\up_h=0$, so $\oposeenelem{0}$ is injective on $\IL(T)$.
Since $\IL(T)$ and $\IQ(T)$ are finite-dimensional and have the same dimension (by construction), $\oposeenelem{0}$ is bijective. 
\end{proof}

\subsubsection{Oseen case \texorpdfstring{$w\neq0$}{w not 0}}\label{sssec:Oseen}

Let us first analyze the role of the convective field $w$ in the local operator $\oposeen{w}$.
To this end we consider the difference between Oseen and Stokes operators, $\oposeen{w}-\oposeen{0}\in\calL(\IX,\IQ')$ and denote its $L^2$-representative (under the identification $\IQ'\cong\IQ$) in the velocity component as the local convection operator
\begin{equation}\tag{$\opconv{w}$} \label{eq:opconv}
\opconv{w}: [\IP^{k}]^d \to [\IP^{k-2}]^d, 
\qquad
 u_h \mapsto \Pi_\Th^{k-2}(w\cdot\nabla u_h),
\end{equation}
with $\opconvelem{w}\cdot = (\opconv{w} \cdot)|_T$ for all $T \in \Th$, 
so that $\langle (\oposeenelem{w}-\oposeenelem{0})\up_h,(r_h,s_h)\rangle_T = (h_T \nu^{-\frac12}\, \opconvelem{w}u_h,r_h)_T$ for all $(r_h,s_h)\in\IQ(T)$, 
where $\Pi_T^{k-2}$ is the $L^2$-orthogonal projection from $[L^2(T)]^d$ to $[\IP^{k-2}(T)]^d$.
\begin{lemma}\label{lem:localconvection}
The linear map $\opconvelem{(\cdot)}: [L^4(T)]^d \to \calL([H^1(T)]^d,[\IP^{k-2}(T)]^d),~ w \mapsto \opconvelem{w}$ is bounded with
\begin{subequations}
    \begin{equation}\label{eq:localconvection}
        \| h_T \nu^{-\frac12}\, \opconvelem{w}u\|_{T} \leq   \nu^{\frac12} |w|_{h,d}\|\nabla u\|_{T},
    \end{equation}
    for all $u_h \in [H^1(T)]^d$, $w \in [L^4(\Omega)]^d$, and $T \in \Th$. With 
    $\contlocal{w}^2 = 2 \contlocal{0}^2 + 2 |w|_{h,d}^2$
    it follows 
    \begin{align} \label{eq:AKw_continuity}
        \|\oposeen{w}\up\|_{\IQ'}
        \le \contlocal{w} \,\dgsnorm{\up},\qquad \up \in \IXreg.
    \end{align}
\end{subequations}
\end{lemma}
\begin{proof}
For any $w \in [L^4(T)]^d$, $u \in [H^1(T)]$ we have by the dual version of the $L^4$-embedding \eqref{l4dual} and H\"older's inequality
\begin{equation} \label{eq:applHldr}
    \| h_T \nu^{-\frac12}\, \opconvelem{w}u\|_{T}
     \leq C_4  \| \nu^{-\frac12} h_T^{1-\frac{d}{4}} w \cdot \nabla u \|_{L^{\frac43}(T)}
     \leq C_4 \nu^{-\frac12} \|h^{1-\frac{d}{4}} w\|_{L^4(T)} \|\nabla u\|_{L^2(T)}.
\end{equation}
Using the definition of $|w|_{h,d}$, we have $C_4\,\nu^{-\frac12}\|h_T^{1-\frac d4}w\|_{L^4(T)}\le \nu^{\frac12}\,|w|_{h,d}$, which concludes the first part of the proof. 
Next, observe that 
\begin{align*}
\|\oposeen{w}\up\|_{\IQ'} \le 
\|\oposeen{0}\up\|_{\IQ'} 
+ \| h \nu^{-\frac12}\, \opconvelem{w} u\|_{\Th},
\end{align*}
where the first term can be bounded by \Cref{lem:AKw_continuity} and the second term by \eqref{eq:localconvection}. 
\end{proof}

Following the analysis in \cite{LLSV_ARXIV_2024}, we can derive invertibility and stability of the local operator $\oposeenelem{w}$ 
on $\IL(T)$ based on the invertibility of the \emph{prototype operator} $\oposeenelem{0}$ on $\IL(T)$ and the proximity of $\oposeenelem{w}$ to $\oposeenelem{0}$ provided that the convection field $w$ is sufficiently small in the $|w|_{h,d}$-norm.

\begin{corollary}[Stable invertibility of $\oposeen{w}$ on $\IL$]\label{cor:AKw}
    Assume that \eqref{ass:wh} holds with the resolution assumption%
    \begin{equation}\tag{A2}\label{ass:A2}
        |w|_{h,d} \leq c_\gamma<\stablocal{0},
    \end{equation}
    where $\stablocal{0}$ is the constant from
    \cref{thm:AK0}. 
    Then the restriction $\oposeen{w}:\IL\to \IQ'$ is invertible and
    \begin{equation}\label{eq:AKw_coerc}
        \|\oposeen{w}\up_h\|_{\IQ'} \ge \stablocal{w}\,\dgnorm{\up_h},
        \qquad \forall \up_h\in \IL,
    \end{equation}
    with $\stablocal{w}:=\stablocal{0} - c_\gamma.$ %
\end{corollary}
\begin{proof}
Using \Cref{lem:localconvection} we get for $T \in \Th$
\[
\|(\oposeenelem{w}-\oposeenelem{0})\up_h\|_{T}
= \| h_T \nu^{-\frac12}\, \opconvelem{w}u_h\|_{T}
\le \nu^{\frac12}\,|w|_{h,d}\,\|\nabla u_h\|_{T} \quad \forall \up_h = (u_h,p_h) \in \IX(T),
\]
and hence
\[
\|(\oposeen{w}-\oposeen{0})\up_h\|_{\IQ'}
\le |w|_{h,d}\,\dgnorm{\up_h} \quad \forall \up_h \in \IX.
\]
Using \cref{thm:AK0} (i.e. $\|\oposeen{0}\up_h\|_{\IQ'}\ge \stablocal{0}\dgnorm{\up_h}$ for all $\up_h \in \IL$) and the resolution assumption, 
we obtain the perturbation bound
\begin{equation*}%
\|(\oposeen{w}-\oposeen{0})\up_h\|_{\IQ'}
\le \frac{|w|_{h,d}}{\stablocal{0}}\|\oposeen{0}\up_h\|_{\IQ'}
\le \gamma\,\|\oposeen{0}\up_h\|_{\IQ'} \quad \forall \up_h \in \IL,
\end{equation*}
with $\gamma := \frac{c_\gamma}{\stablocal{0}} < 1$ by assumption. 
Now, using the Neumann series argument as in \cite[Lemma 3.4]{LLSV_ARXIV_2024} we obtain the result.
\end{proof}

\subsection{Space decomposition}
Building on the local invertibility of $\oposeenelem{w}|_{\IL}$ established in \cref{ssec:local_op} we obtain a global-local space decomposition.
\begin{corollary}\label{cor:local_invertibility}
   The local operator $\oposeen{w}: \IX \to \IQ'$ is bijective on $\IL$, i.e. for all $\rs_h \in \IQ$ there exists a unique $\up_h \in \IL$ such that $\oposeen{w} \up_h = \rs_h$ in $\IQ'$. 
   This implies the space decomposition property 
\begin{equation}\label{eq:splitting}
\IX = \IT_w \oplus \IL,    
\end{equation}
i.e. for all $\up_h \in \IX$ there exist unique $\up_\IT \in \IT_w$ and $\up_\IL \in \IL$ such that $\up_h = \up_\IT + \up_\IL$. 
\end{corollary}
\begin{proof}
  Follows from \cref{thm:AK0} (Stokes case $w=0$) and \cref{cor:AKw} (Oseen case $w\in\IW$).
\end{proof}
\begin{remark}
  The complement space $\IL$ is defined independently of the convection field $w$.
  This is essential for the analysis of the nonlinear Navier--Stokes problem: the same local complement space is used for every iterate in the Picard iteration.
\end{remark}
\begin{remark}[The inverse $\oposeen{w}^{-1}:\IQ'\to \IL$]
The invertibility of $\oposeen{w}$ on $\IL$ allows to define the (bounded) inverse operator 
$\oposeen{w}^{-1}:\IQ'\to \IL$. Then, the splitting $\up_h = \up_\IT + \up_\IL$ can be 
explicitly characterized with $\up_\IL := \oposeen{w}^{-1} \oposeen{w} \up_h \in \IL$ and $\up_\IT := \up_h - \up_\IL \in \IT$.
\end{remark}

\section{The reduced formulation}\label{sec:reduced_formulation}

This section draws two additional structural consequences.
In \cref{ssec:velocity_decomposition} we project the splitting \eqref{eq:splitting} onto the velocity component, yielding a direct-sum decomposition $[\IP^k(\Th)]^d = \IT_w^u \oplus \ILu$ with a proper velocity Trefftz space $\IT_w^u$.
In \cref{ssec:dgreduced_formulation} we recall the reduced formulation from the DG case, setting up the reduced formulation for the Trefftz-DG formulation, which we present in \cref{ssec:tdgreduced_formulation}.
There, we further restrict both spaces to their discretely divergence-free parts, eliminating the saddle-point structure and producing a reduced space on which the problem is coercive --- a key property exploited in the analyses of \cref{sec:oseenanalysis} and the Navier--Stokes case in \cite{SVLL2_ARXIV_2026}.

\subsection{Decomposition of the velocity space}\label{ssec:velocity_decomposition}

In the following, we will see that the decomposition \eqref{eq:splitting} implies a decomposition on the velocity space. 
We define 
\begin{equation}\tag{$\IT_w^u$}\label{eq:ITwu}
    \IT_w^u := \{ u_h \in [\IP^k]^d \mid \text{there exists } p_h^u \in \IP^{k-1} \text{ such that } (u_h,p_h^u) \in \IT_w\}.
\end{equation}
Let us introduce the $L^2$-projected convection-diffusion operator 
\begin{equation}\tag{$\opcd{w}$}\label{eq:opBrmod}
\opcd{w}: [\IP^{k}]^d \to [\IP^{k-2}]^d, v_h \mapsto -\nu \Delta v_h + \Pi^{k-2}(w \cdot \nabla v_h),
\end{equation}
and the $L^2$ projection into the space $\IP^{k-1}(T)$ via gradients
\begin{equation} \tag{$\Pi_{\nabla,T}$} \label{eq:Pinabla}
\Pi_{\nabla,T}: [L^2(T)]^d \to \IP^{k-1}(T)/\mathbb{R}, \qquad v \mapsto \argmin_{q_h \in \IP^{k-1}(T)/\mathbb{R}} \Vert \nabla q_h - v \Vert_{T},
\end{equation}
with $\Pi_{\nabla}: [L^2(\Th)]^d \to \IP^{k-1}/\IP^0$ defined accordingly s.t. $(\Pi_{\nabla} \cdot)|_T = \Pi_{\nabla,T} \cdot$ for all $T\in\Th$.
To express the dependency between pressures and velocities in the Trefftz space we define the higher-order pressure lifting 
\begin{equation} \label{eq:pressure_lifting} \tag{$\Plift{w}$}
\Plift{w}: [H^2(\Th)]^d \to \IP^{k-1}/\IP^0,\quad v_h \mapsto -
\Pi_\nabla \opcd{w} v_h.
\end{equation}
By definition of $\IT_w^u$ we find that for $u_\IT \in \IT_w^u$ there holds $(\opcd{w} + \nabla \Plift{w}) u_\IT = 0$ and hence $\oposeen{w} (u_\IT,\Plift{w} u_\IT) = 0$  (as $\Plift{w} u_\IT$ exactly maps to a pressure $p_h^u$ as in the definition \eqref{eq:ITwu} ).
With these preparations we can now state the splitting of the velocity space.
\begin{lemma}
    There holds 
    \begin{equation} \label{eq:usplitting}
    [\IP^k]^d = \IT_w^u \oplus \ILu, 
    \end{equation}
    i.e.
    to every $u_h \in [\IP^k]^d$ there exist unique $u_\IT \in \IT_w^u$ and $u_\IL \in \ILu$ such that $u_h = u_\IT + u_\IL$.
    Further, the following characterizations hold
    \begin{subequations} \label{eq:Trefftz_velocity_characterization}
    \begin{align}
        \IT_w^u & = \{ u_h \in [\IP^k]^d \mid (u_h,\Plift{w} u_h) \in \IT_w\} \\ 
        \intertext{and the other way around:}
        \IT_w &= \{  (u_h,\Plift{w} u_h + p_h^0) \text{ for } u_h \in \IT_w^u, p_h^0 \in \IP^0\} 
        = (\id,\Plift{w}) \IT_w^u \oplus \{0\} \times \IP^0.
        \label{eq:trefftz_sp_split}
    \end{align}
\end{subequations}
 
\end{lemma}
\begin{proof}
For $\up_h \in \IX$ one easily checks that the following explicit characterization of the splitting $\up_h = \up_\IT + \up_\IL$ holds:
\vspace*{-0.65cm}
$$
\overbrace{(u_h,p_h)}^{\up_h \in \IX} = \overbrace{(u_\IT,\Pi^0 p_h + \Plift{w} u_\IT)}^{\up_\IT \in \IT} + \overbrace{(u_\IL, p_h - \Pi^0 p_h - \Plift{w} u_\IT)}^{\up_\IL \in \IL}.
$$
We observe that the decomposition of $u_h$ into $u_\IT$ and $u_\IL$ is independent of the choice of $p_h$ and hence unique (as $\up_\IT$ and $\up_\IL$ are unique). The characterizations \eqref{eq:Trefftz_velocity_characterization} follow immediately from the definition of $\Plift{w}$ and the above explicit characterization of the splitting which concludes the proof.
\end{proof}

\begin{figure}[!ht]
  \vspace*{-0.5cm}
    \centering
    \newcolumntype{C}{ >{\centering\arraybackslash} m{0.45\textwidth} }
    \newcolumntype{D}{ >{\centering\arraybackslash} m{0.4\textwidth} }
    \begin{tabular}{CcD}
        \includegraphics[width=0.45\textwidth,page=1]{matrix_schematic.pdf}
        &
        $\leadsto$ 
        &
        \includegraphics[width=0.4\textwidth,page=2]{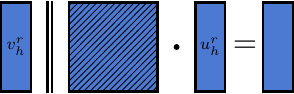} 
    \end{tabular}
    \vspace*{-0.45cm}
    \caption{Schematic of obtaining the reduced problem for the standard DG method. }
    \label{fig:detailed_structure:DG}
\end{figure}

\subsection{Reduced DG formulation}\label{ssec:dgreduced_formulation}
For the analysis of the nonlinear Navier--Stokes problem, it will be crucial to consider the reduced problem, i.e. the problem where a \emph{discrete} divergence-constraint is implemented into the definition of a subspace of $\IX(\mathcal{T}_h)$, cf. \cite{CKS05}. For the standard DG formulation, e.g. \eqref{eq:DG}, this would be the kernel of the discrete divergence operator $\blfdiv(\cdot,\cdot)$ resulting in the reduced DG formulation (for the Oseen problem) only for the velocity (on a subspace of $[\IP^k]^d$):
Find $\rmod{u_h} \in \{ v_h \in [\IP^k]^d \mid \, \blfdiv(v_h,q_h) = 0 \, \forall q_h \in \IP^{k-1} \}$, s.t. 
$$
 \blfvis(\rmod{u_h},\rmod{v_h}) + \blfconv{w}(\rmod{u_h}, \rmod{v_h}) = (f,\rmod{v_h})_{\Th} \qquad \forall \rmod{v_h} \in \{ v_h \in [\IP^k]^d \mid \, \blfdiv(v_h,q_h) = 0 \, \forall q_h \in \IP^{k-1} \}.
$$
See \cref{fig:detailed_structure:DG} for a schematic of the algebraic structure. 
Similarly to the above notation, we will use an underline $\rmod{\cdot}$
to indicate reduced spaces in the following.

\subsection{Reduced Trefftz-DG formulation} \label{ssec:tdgreduced_formulation}
Formulating the reduced problem for the Trefftz-DG method is more subtle, as the discrete divergence constraint is part of the \emph{Trefftz problem} as well as the \emph{local problems}. 

\paragraph{Reduction of the local velocity space.}
We start with the local problem. The divergence constraint is simply the second component of the local operator $\oposeenelem{w}$, i.e. $\Div u_\IL = 0$ for $u_\IL \in \ILu$. Hence, the reduced local velocity space is simply 
\begin{align} \tag{$\LR$}\label{eq:LR}
    \LR(T) &:= \{ u_\IL^{r} \in \ILu(T) : \Div u_\IL^r = 0\}, && \!\!\!\!\!\!\LR = \LR(\Th) \!= \!\! \prod_{T \in \Th} \!\!\LR(T).
    \intertext{
with 
$\dim(\LR) = \dim(\ILu) - \dim(\IP^{k-1})$ as $\Div|_\Th: \ILu \to (\IP^{k-1})'$ is surjective (implied by \cref{cor:AKw}). 
A corresponding reduced local test space acting only on velocities and which renders (pressure) gradient fields $\nabla \IP^{k-1}(T)$ invisible to the local operator is given by the $L^2$ orthogonal complement to  $\nabla \IP^{k-1}(T)$ in $[\IP^{k-2}]^d$. 
We define the reduced local test space as
    }
\tag{$\QR$}\label{eq:QR}
\QR(T) &:= \ker \Pi_{\nabla,T}
\ && \!\!\!\!\!\!\QR = \QR(\Th) \!= \!\!\prod_{T \in \Th} \!\!\QR(T),
\end{align}
so that $\langle r_h, \nabla q_h \rangle_{\Th} = 0$ for all $r_h \in \QR$, $q_h \in \IP^{k-1}$.
The norm is given by the velocity part of the $\IQ$-norm, i.e. for $r_h \in \QR$ we set
\begin{equation*}%
    \|r_h\|_{\QR} := \|r_h\|_{\Th}
\end{equation*}
and the dual norm $\|\cdot\|_{\QR'}$ is defined accordingly.
The dimension of the reduced test space matches the reduced local velocity space, i.e. there holds $\dim(\QR) = \dim([\IP^{k-2}]^d) - \dim(\nabla \IP^{k-1}) = \dim(\LR)$.
We define the reduced local operator $\oposeenrelem{w}: [H^2(T)]^d \to \QR(T)'$ as %
\begin{equation} \tag{$\oposeenrelem{w}$}\label{eq:ArKw}
\langle \oposeenrelem{w} \rmod{u}_\IL, r_h \rangle_T := \langle h \nu^{-\frac12} (-\nu \Delta \rmod{u}_\IL + w \cdot \nabla \rmod{u}_\IL), r_h \rangle_T 
\qquad \forall \rmod{u}_\IL \in [H^2(T)]^d, r_h \in \QR(T).
\end{equation}
Correspondingly, we define $\oposeenr{w}: [H^2(\Th)]^d \to \QR'$. Note that the difference between $\opcd{w}$ and $\oposeenr{w}$ is in the treatment of arising gradient fields and the scaling. 

\begin{corollary}[Continuity and stability of $\oposeenr{w}$ on $\LR$]\label{cor:AKrw}
    Assume that \eqref{ass:wh} holds with the resolution assumption \eqref{ass:A2}. 
    Then $\oposeenr{w}:\LR \to \QR'$ is invertible.
    With $\stablocal{w}$, $\contlocal{w}$ as in \cref{lem:AKw_continuity,cor:AKw}, for all $\rmod{u}_\IL\in \LR$ and $\ur \in [H^2(\Th)]^d$ there holds
    \begin{align}
        \stablocal{w} \, \nu^{\frac12} \vdgnorm{\rmod{u}_\IL} \le \|\oposeenr{w} \rmod{u_\IL} \|_{\QR'}, \qquad \|\oposeenr{w} \rmod{u} \|_{\QR'} &\le \contlocal{w}\, \nu \, \vdgsnorm{\ur}. \label{eq:AKrw_cont_coer}
    \end{align}
\end{corollary}
\begin{proof}
    By stable invertibility of $\oposeen{w}$, cf.\ \cref{cor:AKw}, there is unique $\up_\IL = (u_\IL,p_\IL) \in \IL$ so that 
    $\oposeen{w} \up_\IL = (\oposeenr{w} \rmod{u}_\IL,0)$. 
    The corresponding full local solution has velocity component $u_\IL=\rmod{u}_\IL$,
    while the pressure component is the unique $p_\IL\in\ILp$ such that
    \[
    \oposeen{w}(u_\IL,p_\IL)=(\oposeenr{w}\rmod{u}_\IL,0)\quad\text{in }\IQ'.
    \]
    and hence
    \begin{equation*}
        \|\oposeenr{w} \rmod{u}_\IL \|_{\QR'}
          = \|\oposeen{w} \up_\IL \|_{\IQ'}
          \ge \stablocal{w} \Vert \up_\IL \Vert_{\IX} \ge 
          \stablocal{w} \nu^{\frac12} \vdgnorm{u_\IL} = 
          \stablocal{w} \nu^{\frac12} \vdgnorm{\ur_\IL}.
    \end{equation*}
    Similarly, we find 
    \[
        \|\oposeenr{w} \ur \|_{\QR'} \le \|\oposeen{w}(\ur,0)\|_{\IQ'},
    \]
    which concludes the proof after applying the continuity estimate \eqref{eq:AKw_continuity} to $\up = (\ur,0)$.
\end{proof}    

\paragraph{Reduction of the Trefftz velocity space}
We now carry out the analogous reduction for the Trefftz space.
By \eqref{eq:trefftz_sp_split} from \cref{ssec:velocity_decomposition}, the Trefftz space decomposes as
$$
\IT_w = (\id,\Plift{w}) \IT_w^u \oplus \{0\} \times \IP^0.
$$
The form $\blfoseen{w}$ is coercive on the first component $(\id,\Plift{w})\IT_w^u$ \cite{LLS_NM_2024}, while it vanishes identically on the second component $\{0\}\times\IP^0$, since no coupling between pressure-only test and trial functions arises. The two components interact solely through $\blfdiv$, so the global problem has the standard saddle-point structure in which $\{0\}\times\IP^0$ serves as the Lagrange multiplier space enforcing the discrete divergence constraint; cf.\ \cref{fig:detailed_structure:TDG}.

\begin{lemma}[LBB stability]\label{lem:LBB_P0}
Assume \eqref{ass:wh} and \eqref{ass:A2}.
There exists a constant $\alpha_0>0$, depending only on $\Omega$, $k$, and the shape-regularity of $\Th$, such that
\begin{equation}\label{eq:lbb_P0}
    \sup_{u_\IT \in \IT_w^u \setminus\{0\}} \frac{\blfdiv(u_\IT,\,p_h^0)}{\|u_\IT\|_{1,h}}
    \;\ge\; \alpha_0\,\|p_h^0\|_{0,h}
    \qquad \forall\, p_h^0 \in \IP^0.
\end{equation}
\end{lemma}
\begin{proof}
Following \cite[Lemma 7]{LLS_NM_2024}, to $p_h^0 \in \IP^0$ there is $u_h^p \in [\IP^1]^d \cap \{ \Div u_h^p|_\Th = 0 \}$ s.t. $\blfdiv(u_h^p,p_h^0) = \Vert p_h^0\Vert_{0,h}^2$ with $\Vert u_h^p \Vert_{1,h} \leq C \Vert p_h^0\Vert_{0,h}$. %
    As $u_h^p \in [\IP^{k}]^d$, we can uniquely decompose it into $u_h^p = u_\IT + u_\IL$ with $u_\IT \in \IT_w^u$ and $u_\IL \in \IL^u$. As further $\Div u_\IT|_{\Th} = 0$ we also have $\Div u_\IL|_{\Th} = 0$ and with \cref{eq:divrelation} we have $\blfdiv(u_\IL,p_h^0) = (\jmp{u_\IL \cdot n},\avg{p_h^0})_{\Fh} = 0$ and hence $\blfdiv(u_\IT,p_h^0) = \blfdiv(u_h^p,p_h^0) - \blfdiv(u_\IL,p_h^0) = \Vert p_h^0 \Vert_{0,h}^2$.
    Further, we have $\oposeenr{w} u_\IT = 0$ and hence $\oposeenr{w} u_\IL = \oposeenr{w} u_h^p$ which yields $\Vert u_\IL \Vert_{1,h} \leq \Vert \oposeenr{w}^{-1} \Vert \Vert \oposeenr{w} \Vert  \Vert u_h^p \Vert_{1,h}$ and hence
    $$
\Vert u_\IT \Vert_{1,h}
\leq 
\Vert u_h^p \Vert_{1,h} + 
\Vert u_\IL \Vert_{1,h}
\leq (1 + \frac{\contlocal{w}}{\stablocal{w}}) C 
\Vert p_h^0 \Vert_{0,h}. \vspace*{-4ex}
    $$
\end{proof}
According to this structure, we set the notation of 
\emph{discretely divergence-free} functions as those functions $u_h \in [\IP^k]^d$ that are element-wise divergence-free, i.e. $\Div u_h|_T = 0$ for all $T \in \Th$, and satisfy the discrete divergence constraint against $\IP^0$ pressure test functions, i.e. $\blfdiv(u_h, q_h^0) = 0$ for all $q_h^0 \in \IP^0$: 
\begin{equation}\tag{$\XR$}\label{eq:XR}
\XR := \{ u_h \in [\IP^k]^d \mid\, \blfdiv(u_h,q_h^0) = 0 ~ \forall q_h^0 \in \IP^0 \text{ and } \Div|_{\Th} u_h =0 \}.
\end{equation}
In analogy to $\IXreg$ we introduce the notation $\XRreg := [H^2(\Th)]^d$.
The subspace of $\XR$ of \emph{discretely divergence-free} velocity Trefftz functions is then
\begin{equation}\tag{$\TR_w$}\label{eq:TR}
\TR_w = \{ \rmod{u_\IT} \in \IT_w^u : \blfdiv(\rmod{u_\IT}, q_h^0) = 0 ~\forall q_h^0 \in 
\IP^0\} \subset \IT_w^u.
\end{equation}
We obviously have $\LR \subseteq \XR$ and $\TR_w \subseteq \XR$.
Dimension counting shows that
\begin{align*}
& \dim(\TR_w) + \dim(\LR)  = \dim(\IT_w^u) - \dim(\IP^0) + \dim(\ILu) - \dim(\IP^{k-1}) \\
& = \dim([\IP^k]^d) - \dim([\IP^{k-2}]^d) - 2 \dim(\IP^0) + 
\dim([\IP^{k-2}]^d) + \dim(\IP^0) - \dim(\IP^{k-1})
= \dim(\XR),
\end{align*}
and hence we also have the space decomposition of the discretely divergence-free velocities into Trefftz velocities in $\TR_w$ and complement $\LR$:
\begin{equation}\label{eq:XRdecomp}
\TR_w \oplus \LR = \XR.    
\end{equation}

\begin{figure}[htbp!]
 \centering
\vspace*{-0.325cm}
 \newcolumntype{C}{ >{\centering\arraybackslash} m{0.49\textwidth} }
\newcolumntype{D}{ >{\centering\arraybackslash} m{0.4\textwidth} }
\begin{tabular}{CcD}
 \includegraphics[width=0.49\textwidth,page=4]{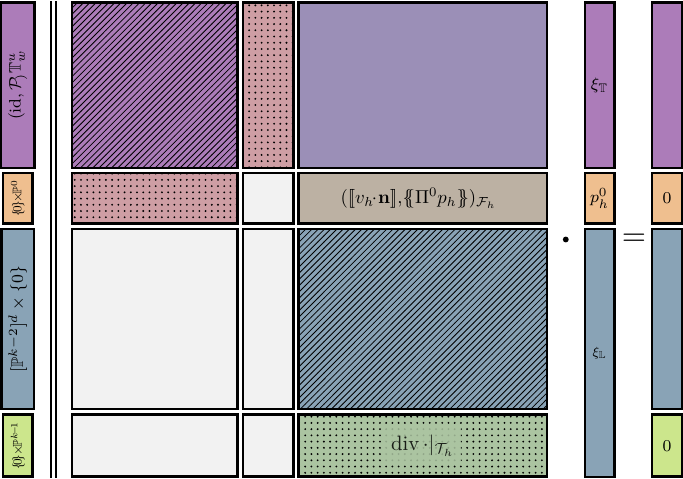}
 &
 $\leadsto$ 
 &
 \includegraphics[width=0.4\textwidth,page=5]{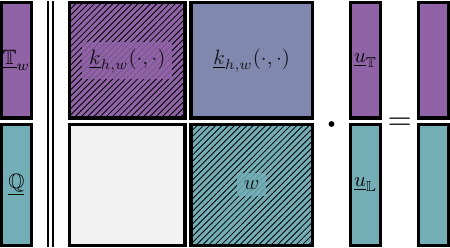} 
\end{tabular}
\vspace*{-0.325cm}
\caption{Algebraic structure before and after reformulation as a reduced problem on $\XR$.}
 \label{fig:detailed_structure:TDG}
 \end{figure}

Next, we want to use this reduction to formulate a Trefftz-DG method on $\XR$ that yields an equivalent velocity solution to the formulation \eqref{eq:tdgoseen}.

Find $\ur_h = \ur_\IT + \ur_\IL \in \XR$ such that
for all $\zr_h = (\vr_\IT,\rr_h) \in \ZR := \TR_w \times \QR$ there holds
\begin{align}\label{eq:redtdgoseen}
    \blftrefftzr{w}(\ur_h,\zr_h) & =\rmod{\ell}_h(\vr_\IT,\rr_h),
    \text{ with } \blftrefftzr{w}(\ur_h,\zr_h) := \blfoseenr{w}(\ur_h,\vr_\IT) + \langle \oposeenr{w} \ur_h, \rr_h \rangle_{\Th},~~ %
    \tag{$\blftrefftzr{w}$}
    \\
    \text{ where }
\blfoseenr{w}(\ur,\vr) & := \blfoseen{w}((\ur,\Plift{w} \ur),(\vr, \Plift{w} \vr))
    \tag{$\blfoseenr{w}$} \\
    \text{ and }
\rmod{\ell}_h(\vr,\rr)
&:=
(f,\vr)_\Th-\blfdiv(\vr,\Pi_\nabla f)+(f,h\nu^{-\frac12}\rr)_\Th.
\tag{$\rmod{\ell}_h$}
\label{eq:rhs_functional}
\end{align}

\begin{theorem}[Equivalence of \eqref{eq:tdgoseen} and \eqref{eq:redtdgoseen}]\label{cor:uniquesol}
Under the assumptions of \cref{cor:AKw}, the two formulations are equivalent:
any solution $(u_h, p_h) \in \IX$ of \eqref{eq:tdgoseen} has $u_h \in \XR$ satisfying
\eqref{eq:redtdgoseen}, and conversely, any solution $\ur_h \in \XR$ of
\eqref{eq:redtdgoseen} is the velocity component of a solution of \eqref{eq:tdgoseen}.
\end{theorem}
\begin{proof}
  Let $\up_h = (u_h,p_h) \in \IX$ be solution to \eqref{eq:tdgoseen}.
  We decompose $\up_h = \up_\IT + \up_\IL$, i.e 
  $(u_h,p_h) = (u_\IT,p_\IT) + (u_\IL,p_\IL)$ and have from the second equation of 
  the local problem 
  that 
  $\Div u_h|_\Th = 0$. With the second Trefftz condition, we hence have
  $\Div u_h|_\Th = \Div u_\IT|_\Th =\Div u_\IL|_\Th =0$. Further,
  testing \eqref{eq:tdgoseen} with $(0,q_h^0)$ for $q_h^0\in\IP^0$ we find -- together with \cref{eq:divrelation} -- that we have $\blfdiv(u_h,q_h^0) = \blfdiv(u_\IT,q_h^0) = \blfdiv(u_\IL,q_h^0) = 0$.
  Hence, $u_h, u_\IL, u_\IT \in \XR$ and $\up_\IT = (u_\IT, \Plift{w} u_\IT + p_h^0)$, $p_h^0\in\IP^0$ so that $u_\IT \in \TR_w$ and $u_\IL \in \LR$.
  Writing out the first equation of the  Trefftz condition
  we get
  $$
   \opcd{w} u_h + \nabla p_h = \opcd{w} u_\IL + \nabla p_\IL = \Pi^{k-2} f.
  $$
  With $u_\IL \in \LR$ so that $\oposeenr{w} u_\IL = h \nu^{-\frac12} f$ in $\QR'$ and
  $p_\IL = \Pi_\nabla (f - \opcd{w} u_\IL) = \Pi_\nabla f + \Plift{w} u_\IL$ we found an expression for the unique solution $\up_\IL$ (uniqueness followed from \cref{cor:AKw}).
  We conclude $\xi_\IT = (u_\IT,\Plift{w} u_\IT + p_h^0)$ for $u_\IT \in \TR_w,~p_h^0 \in \IP^0$ and $\xi_\IL = (u_\IL, \Plift{w} u_\IL) + (0,\Pi_\nabla f)$ for $u_\IL \in \LR$, hence $\up_h = (u_h, \Plift{w} u_h) + (0,\Pi_\nabla f + p_h^0)$.

  Now, consider Trefftz test function $\vq_\IT \in \IT_w$ with $\vq_\IT = (v_\IT, \Plift{w} v_\IT)$ for $v_\IT \in \TR_w$ and plug these into \eqref{eq:tdgoseen}:
  $\blfoseen{w}(\up_h,\vq_\IT) =  (f, v_\IT)_\Th$.
  Plugging in the previously obtained characterizations for $\up_\IL$ we find 
  \begin{align}
     \blfoseen{w}(\up_h,\vq_\IT)
     & = \blfoseen{w}((u_h,\Plift{w} u_h),(v_\IT,\Plift{w} v_\IT))
    + \blfoseen{w}((0,\Pi_\nabla f+p_h^0),(v_\IT,\Plift{w} v_\IT)) 
    \nonumber \\
    & = \blfoseenr{w}(u_h,v_\IT) 
    + \blfdiv(v_\IT,\Pi_\nabla f) + \blfdiv(v_\IT,p_h^0) = (f, v_\IT)_\Th \qquad \forall v_\IT \in \TR_w.
    \label{equiv1}
  \end{align}
  As $v_\IT \in \TR$ we have $\blfdiv(v_\IT,p_h^0) = 0$ and hence, the velocity solution of \eqref{eq:tdgoseen} also solves \eqref{eq:redtdgoseen}. 
  To go the other way around let $\ur_h = \ur_\IT + \ur_\IL \in \XR$ with $\ur_\IT \in \TR_w$, $\ur_\IL \in \LR$ be the solution to \eqref{eq:redtdgoseen}. 
  From the considerations above we know already that $\up_\IL = (\ur_\IL, \Plift{w} \ur_\IL + \Pi_\nabla f)$ solves the local problem of \eqref{eq:tdgoseen}.
  For $\up_\IT$ we set $\up_\IT = (\ur_\IT, \Plift{w} \ur_\IT + p_h^0)$, yielding $\up_h = ( \ur_h, \Plift{w} \ur_h + p_h^0 + \Pi_\nabla f)$ for an undetermined $p_h^0\in\IP^0$. 
  Plugging in yields \eqref{equiv1}, but this time $\blfdiv(v_\IT,p_h^0)$ does not vanish as we only have $v_\IT \in \IT_w^u$ and not $v_\IT \in \TR_w$. However, due to the LBB condition \cref{lem:LBB_P0} we know that there is a unique $p_h^0 \in \IP^0$ so that \eqref{equiv1} indeed holds.
  This concludes the proof.
\end{proof}

\subsubsection{Norm equivalence between velocity and lifted velocity-pressure pair}

Let us define the scaled velocity norm
$\norm{\cdot}_{\XR} := \nu^{\frac12}\vdgnorm{\cdot}$ 
on $[H^1(\Th)]^d$.
and
$\norm{\cdot}_{\XR,\ast} := \nu^{\frac12}\vdgsnorm{\cdot}$
on $\XRreg = [H^2(\Th)]^d$.
A key auxiliary fact, used throughout, is that the pressure lifting satisfies
$\Pi^0_\Th \Plift{w} u = 0$ for any $u\in[H^2(\Th)]^d$
(since $\Plift{w}$ maps into $\IP^{k-1}/\IP^0$), and by \eqref{eq:AKw_continuity} for any $\ur \in \XRreg$
\begin{subequations}
\begin{align}
  \norm{\ur}_{\XR,\ast}^2 &\le \dgsnorm{(\ur,\Plift{w} \ur)}^2
  = \norm{\ur}_{\XR,\ast}^2 + \nu^{-1}\|h\nabla\Plift{w} \ur\|_\Th^2
  \le \norm{\ur}_{\XR,\ast}^2 + \|\nu^{-\frac12} h \opcd{w} \ur \|_\Th^2 \nonumber \\
  &
  \le \norm{\ur}_{\XR,\ast}^2 + \| \oposeen{w} (\ur,0)\|_{\IQ'}^2 
  \stackrel{\eqref{eq:AKw_continuity}}{\le} \norm{\ur}_{\XR,\ast}^2 + \contlocal{w}^2 \| (\ur,0)\|_{\IX,\ast}^2 
   \le (1+ \contlocal{w}^2) 
  \norm{\ur}_{\XR,\ast}^2. \label{eq:plift_norm}
\end{align}
On the discrete space $\XR$ we have by norm equivalence of $\vdgsnorm{\cdot}$ and $\vdgnorm{\cdot}$ (i.e. $\norm{\cdot}_{\XR,\ast}$ and $\norm{\cdot}_{\XR}$)
\begin{equation}
\norm{\ur_h}_{\XR} \le \dgnorm{(\ur_h,\Plift{w} \ur_h)} \le C (1+\contlocal{w}^2)^{\frac12} \norm{\ur_h}_{\XR} \text{ for all } \ur_h \in \XR.
\label{eq:plift_normequiv}
\end{equation}
\end{subequations}

\section{Analysis of the Oseen Trefftz-DG method}\label{sec:oseenanalysis}\label{sec:analysis}

We establish the three fundamental stability properties of the reduced bilinear form
$\blftrefftzr{w}$ from \eqref{eq:redtdgoseen}: continuity, 
coercivity of the global part $\blfoseenr{w}$ on the Trefftz subspace $\TR_w$, and an inf-sup condition.  
The last result yields well-posedness of
\eqref{eq:redtdgoseen}.
We equip the test space $\ZR = \TR_w\times\QR$ with the norm
\begin{equation}\label{eq:ZRnorm}\tag{$\|\cdot\|_{\ZR}$}
  \|(\vr_\IT,\rr_h)\|_{\ZR}^2 := \|\vr_\IT\|_{\XR}^2 + \|\rr_h\|_{\QR}^2.
\end{equation}

\subsection{Continuity}\label{ssec:oseen_cont}
\begin{lemma}[Continuity]\label{lem:blfcont}
Assume \eqref{ass:wh}. There holds
\begin{subequations}
  \begin{align}
    \blfoseen{w}(\up,\vq_h) &\leq
    \contoseen{w} \norm{\up}_{\IX,\ast} \norm{\vq_h}_{\IX}, \hspace*{-0.4cm}& (u,p) = \xi &\in \IXreg, \hspace*{-0.4cm}& (v_h,p_h) = \vq_h &\in \IX, 
    \label{eq:cont_blfoseen}
    \\ 
    \blftrefftz{w}(\up,\z_h) &\leq
    \conttrefftz{w} \norm{\up}_{\IX,\ast} \norm{\z_h}_{\IZ}, \hspace*{-0.4cm}& (u,p) = \xi &\in \IXreg, \hspace*{-0.4cm}&(\vq_\IT,\rs_h) = \z_h &\in \IZ_w, 
    \label{eq:cont_blfk}
    \\
    \blfoseenr{w}(\ur,\vr_h) &\leq
    \contoseenr{w} \norm{\ur}_{\XR,\ast} \norm{\vr_h}_{\XR}, \hspace*{-0.4cm}& \ur &\in \XRreg, \hspace*{-0.4cm}&\vr_h &\in \XR,
    \label{eq:cont_blfoseenr}
    \\ 
    \blftrefftzr{w}(\ur,\zr_h) &\leq 
    \conttrefftzr{w} \norm{\ur}_{\XR,\ast} \norm{\zr_h}_{\ZR}, \hspace*{-0.4cm}& \ur &\in \XRreg, \hspace*{-0.4cm}&(\vr_\IT,\rr_h) = \zr_h &\in \ZR_w,
    \label{eq:cont_blfkr}
    \\
    \blfoseenr{w}(\ur_\IL,\ur_\IT) &\le
    \contcross{w}\,\|\ur_\IL\|_{\XR}\,\|\ur_\IT\|_{\XR},
    \hspace*{-0.4cm}& \ur_\IL &\in \LR, \hspace*{-0.4cm}&\ur_\IT &\in \TR_w.
    \label{eq:cross_improved}
  \end{align}
\end{subequations}
where the constants $\contoseen{w}$, $\conttrefftz{w}$, $\contoseenr{w}$, $\conttrefftzr{w}$, depend on the convection field through the ratio $|w|_{1,4,h}/\nu$.
In contrast, $\contcross{w}$ depends on the convection field only through the resolution quantity $|w|_{h,d}$, and not through $|w|_{1,4,h}/\nu$.
\end{lemma}
\begin{proof}
Bounding $\blfvis(u,v_h) \leq \contvis \nu \vdgsnorm{u} \vdgnorm{v_h}$, $\blfconv{w}(u,v_h) \leq \contconv |w|_{1,4,h}\|u\|_{1,h,\ast}\|v_h\|_{1,h}$ (using \eqref{eq:ch_cont}), and $\blfdiv(u,q_h) \leq \contdiv  \vdgnorm{u} \norm{q_h}_{0,h}$, $\blfdiv(v_h,q) \leq \contdiv  \vdgnorm{v_h} \norm{q}_{0,h}$, we obtain the first bound \eqref{eq:cont_blfoseen} with $ \contoseen{w} = C \Bigl(\contvis+\contdiv + \contconv \frac{|w|_{1,4,h}}{\nu}\Bigr). $
The second bound \eqref{eq:cont_blfk} then follows from \eqref{eq:cont_blfoseen} and \eqref{eq:AKw_continuity} with $\conttrefftz{w} = \contoseen{w} + \contlocal{w}$.
For \eqref{eq:cont_blfoseenr} we find
$$
\blfoseenr{w}(\ur,\vr_h) = 
\blfoseen{w}((\ur,\Plift{w}\ur),(\vr_h,\Plift{w}\vr_h))
\le \contoseen{w}
\norm{(\ur,\Plift{w}\ur)}_{\IX,\ast} \norm{(\vr_h,\Plift{w}\vr_h)}_{\IX}
\le \contoseenr{w}
\norm{\ur}_{\XR,\ast} \norm{\vr_h}_{\XR}
$$ 
with $\contoseenr{w} = C (1 + \contlocal{w}^2) \contoseen{w}$. Finally, for \eqref{eq:cont_blfkr} we gather \eqref{eq:cont_blfoseenr} and \eqref{eq:AKrw_cont_coer} so that $\conttrefftzr{w} = \contoseenr{w} + \contlocal{w}$.

Finally, since $\ur_\IL\in\LR\subset\ILu$ and $\ILu$ is an admissible choice for
$\IX_{\rm ps}(\Th)$ by \cref{rem:ILu_poincare}, the improved convection estimate
\eqref{eq:ch_mixed_L4} yields
\[
|\blfconv{w}(\ur_\IL,\ur_\IT)|
\le 
C_{t,\ast}\,\nu\,|w|_{h,d}\,\|\ur_\IL\|_{1,h}\,\|\ur_\IT\|_{1,h}
= 
C_{t,\ast}\,|w|_{h,d}\,\|\ur_\IL\|_{\XR}\,\|\ur_\IT\|_{\XR},
\]
with
$\contcross{w} := C\bigl(\contvis+\contdiv\,\contlocal{w}+C_{t,\ast}|w|_{h,d}\bigr).$
\end{proof}

\subsection{Stability of the reduced formulation}
\label{ssec:oseen_reduced_stability}

We first establish stability for the reduced velocity formulation.
The stability proof reflects the space decomposition $\XR=\TR_w\oplus\LR$. 
The $\TR_w$-component is controlled by coercivity of the reduced Oseen form $\blfoseenr{w}$, whereas the $\LR$-component is controlled by the local residual operator $\oposeenr{w}$.
The interaction between the two components is bounded by the
improved cross estimate \eqref{eq:cross_improved}.

The first result, \cref{lem:coerc_IT_perp}, is therefore a coercivity statement on the reduced Trefftz kernel.
The second result, \cref{th:Khinfsup}, then combines this coercivity with the local stability on $\LR$ and yields the inf-sup condition for the reduced coupled form $\blftrefftzr{w}$ on $\XR\times\ZR$.
Finally, \cref{lem:oseen_energy_stab}, discusses the stability of the solution map, another crucial step in preparation for the analysis of the nonlinear problem in \cite{SVLL2_ARXIV_2026}.

We begin with coercivity on $\TR_w$. The penalty parameter is chosen large enough to absorb the pressure-lifting contribution generated by $\Plift{w}$.

\begin{lemma}[Coercivity of $\blfoseenr{w}$ on $\TR_w$]\label{lem:coerc_IT_perp}
There exist $\lambdadg>0$ and $C>0$, depending only on shape-regularity, $k$, 
such that for the IP parameter $\lambda\ge\lambdadg+C\contlocal{w}^2$ in \eqref{eq:sipdg}, there holds
\begin{equation}\label{eq:coerc_TR}
  \blfoseenr{w}(\ur_\IT,\ur_\IT) \ge \staboseenr{w} \norm{\ur_\IT}_{\XR}^2
  \qquad\forall\,\ur_\IT\in\TR_w,
  \quad\text{with } \staboseenr{w}=\tfrac12. %
\end{equation}
\end{lemma}
\begin{proof}
Let $\up_h := (\ur_\IT,\Plift{w}\ur_\IT)\in\IT_w$.
By \cref{ass:ch}, $\blfconv{w}(\ur_\IT,\ur_\IT)\ge 0$, we have
\[
  \blfoseenr{w}(\ur_\IT,\ur_\IT)
  = \blfoseen{w}(\up_h,\up_h)
  \ge \blfvis(\ur_\IT,\ur_\IT) + 2\blfdiv(\ur_\IT,\Plift{w} \ur_\IT).
\]
The SIPG form satisfies $\blfvis(\ur_\IT,\ur_\IT)\ge \frac34 \|\ur_\IT\|_{\XR}^2
+ (\lambda-\lambdadg)\nu\|h^{-\frac12}\jmp{\ur_\IT}\|_{\Fh}^2$ for $\lambda\ge\lambdadg$. We will use the latter contribution to bound the $\blfdiv(\cdot,\cdot)$ contribution.
Since $\Div\ur_\IT=0$ on each $T$,
\begin{align*}
2 \blfdiv(\ur_\IT,\Plift{w} \ur_\IT) &= 2 (\jmp{\ur_\IT\cdot n},\avg{\Plift{w}\ur_\IT})_\Fh
\leq 
\varepsilon^{-1} \nu\|h^{-\frac12}\jmp{\ur_\IT}\|_{\Fh}^2
+
\varepsilon \nu^{-1} \| h^{\frac12}\avg{\Plift{w}\ur_\IT}\|_{\Fh}^2,
\end{align*}
where we used Cauchy--Schwarz and Young with $\varepsilon>0$. With 
the polynomial trace inequality, discrete norm equivalences and scaling arguments and the Trefftz identity as in \eqref{eq:plift_norm} we obtain
\begin{align*}
\nu^{-1}\| h^{\frac12}\avg{\Plift{w}\ur_\IT}\|_{\Fh}^2
\leq 
C \nu^{-1} \|\Plift{w} \ur_\IT\|_{\Th}^2 
\leq 
C \nu^{-1} \|h \nabla \Plift{w} \ur_\IT\|_{\Th}^2 
\stackrel{\eqref{eq:plift_norm}}{\le}
C \contlocal{w}^2 \|\ur_\IT\|_{\XR}^2. 
\end{align*}
With $\varepsilon = (4C\contlocal{w}^2)^{-1}$,
the claim then follows with $\lambda \ge \lambdadg + 4C\contlocal{w}^2$ as
\[
  \blfoseenr{w}(\ur_\IT,\ur_\IT)
  \ge 
  \frac34 \|\ur_\IT\|_{\XR}^2
+ (\lambda-\lambdadg)\nu\|h^{-\frac12}\jmp{\ur_\IT}\|_{\Fh}^2
- \frac14 \|\ur_\IT\|_{\XR}^2 
- 4C\contlocal{w}^2 \nu\|h^{-\frac12}\jmp{\ur_\IT}\|_{\Fh}^2.
\vspace*{-0.4cm}
\]
\end{proof}

We now extend the preceding coercivity estimate from $\TR_w$ to the whole space $\XR$.
The proof tests the Trefftz component by itself and the complement component by the Riesz representative of its local residual.

\begin{theorem}[Inf-sup stability of $\blftrefftzr{w}$]
\label{th:Khinfsup}
Assume \eqref{ass:wh} and \eqref{ass:A2}, and let $\lambda$ be as in
\cref{lem:coerc_IT_perp}. For all $\ur_h\in\XR$, there holds
\begin{equation}\label{eq:Khr_infsup}
  \sup_{\zr_h\in\ZR\setminus\{0\}}
  \frac{\blftrefftzr{w}(\ur_h,\zr_h)}{\|\zr_h\|_{\ZR}}
  \;\ge\; \stabtrefftzr{w}\|\ur_h\|_{\XR},
\end{equation}
with
$
\delta := \max\left\{ 1,\, \stablocal{w}^{-1},\, \frac{\contcross{w}^{2}} {\staboseenr{w}\stablocal{w}^{2}} \right\},
\
\stabtrefftzr{w} := \frac{\staboseenr{w}}{2\sqrt{2}\,\delta},
$
where $\staboseenr{w}=\tfrac12$ is the coercivity constant from \cref{lem:coerc_IT_perp} and $\contcross{w}$ is the constant from \eqref{eq:cross_improved}.
Again, $\stabtrefftzr{w}$ depends on the convection field only through the resolution quantity $|w|_{h,d}$, and not through $\|w\|_{1,h}/\nu$.
\end{theorem}
\begin{proof}
Decompose $\ur_h = \ur_\IT + \ur_\IL\in\XR$ with $\ur_\IT\in\TR_w$ and
$\ur_\IL\in\LR$, and choose the test function
\[
  \zr_h := \bigl(\ur_\IT,\delta \rr_h\bigr)\in\ZR.
\]
and where $\rr_h\in\QR$ is the Riesz representative of
$\oposeenr{w}\ur_\IL\in\QR'$, i.e.
\[
\langle\oposeenr{w}\ur_\IL,\rr_h\rangle_\Th
=
\|\oposeenr{w}\ur_\IL\|_{\QR'}^2
=
\|\rr_h\|_{\QR}^2.
\]
Since $\ur_\IT\in\TR_w$, we have $\langle \oposeenr{w}\ur_\IT,\rr_h\rangle_\Th=0$.

Using coercivity of $\blfoseenr{w}$ on $\TR_w$ (\cref{lem:coerc_IT_perp}), the cross-coupling bound \eqref{eq:cross_improved}, the stability estimate from \cref{cor:AKrw}, and Young's inequality, we obtain
\begin{align*}
  \blftrefftzr{w}(\ur_h,\zr_h)
  &=
  \blfoseenr{w}(\ur_\IT,\ur_\IT)
  + \blfoseenr{w}(\ur_\IL,\ur_\IT)
  + \delta\langle\oposeenr{w}\ur_\IL,\rr_h\rangle_\Th
  \\
  &\ge
  \staboseenr{w}\|\ur_\IT\|_{\XR}^{2}
  -
  \contcross{w}\|\ur_\IL\|_{\XR}\|\ur_\IT\|_{\XR}
  +
  \delta\|\oposeenr{w}\ur_\IL\|_{\QR'}^{2}
  \\
  &\ge
  \staboseenr{w}\|\ur_\IT\|_{\XR}^{2}
  -
  \frac{\contcross{w}}{\stablocal{w}}
  \|\oposeenr{w}\ur_\IL\|_{\QR'}\|\ur_\IT\|_{\XR}
  +
  \delta\|\oposeenr{w}\ur_\IL\|_{\QR'}^{2}
  \\
  &\ge
  \frac{\staboseenr{w}}{2}\|\ur_\IT\|_{\XR}^{2}
  +
  \left(
  \delta
  -
  \frac{\contcross{w}^{2}}
       {2\staboseenr{w}\stablocal{w}^{2}}
  \right)
  \|\oposeenr{w}\ur_\IL\|_{\QR'}^{2}
  \\
  &\ge
  \frac{\staboseenr{w}}{2}\|\ur_\IT\|_{\XR}^{2}
  +
  \frac{\delta}{2}
  \|\oposeenr{w}\ur_\IL\|_{\QR'}^{2}.
  =
  \frac{\staboseenr{w}}{2\delta}
  \|\zr_h\|_{\ZR}^{2}.
\end{align*}
Moreover, since $\delta\stablocal{w}\ge1$, the stability estimate from
\cref{cor:AKrw} gives
\[
  \|\zr_h\|_{\ZR}^{2}
  =
  \|\ur_\IT\|_{\XR}^{2}
  +
  \delta^{2}\|\oposeenr{w}\ur_\IL\|_{\QR'}^{2}
  \ge
  \|\ur_\IT\|_{\XR}^{2}
  +
  \|\ur_\IL\|_{\XR}^{2}
  \ge
  \frac12\|\ur_h\|_{\XR}^{2}.
\]
Taking the supremum over $\zr_h\in\ZR\setminus\{0\}$ proves \eqref{eq:Khr_infsup}.
\end{proof}

\begin{lemma}[Stability of the reduced Oseen solution map]
\label{lem:oseen_energy_stab}
Let $w_h\in \IW$ and let \eqref{ass:wh} and \eqref{ass:A2} hold, such that 
the reduced Oseen problem \eqref{eq:redtdgoseen}
admits a unique solution denoted by $\ur_h$.
Then
\begin{equation}
\|\ur_h\|_{\XR}
\le \frac{C_F}{\stabtrefftzr{w_h}}\nu^{-\frac12}\|f\|_{\Th},
\end{equation}
where $ C_F:=\Bigl((C_\Omega+\contdiv h)^2+h^2\Bigr)^{1/2}. $
In particular, the dependence on the convection field enters only through the
reduced inf-sup constant $\stabtrefftzr{w_h}$.
\end{lemma}
\begin{proof}
Let $\zr_h=(\vr_\IT,\rr_h)\in\ZR=\TR_{w_h}\times\QR$.
The right-hand side of \eqref{eq:redtdgoseen} is
\[
(f,\vr_\IT)_\Th - \blfdiv(\vr_\IT,\Pi_\nabla f)
+ (f,h\nu^{-\frac12}\rr_h)_\Th .
\]
We bound the three terms separately.
For the first term, using discrete Poincar\'e inequality \eqref{sp},
\[
|(f,\vr_\IT)_\Th|
\le \|f\|_{\Th} \|\vr_\IT\|_{\Th}
\le C_\Omega \|f\|_{\Th} \|\vr_\IT\|_{1,h}
= C_\Omega \nu^{-\frac12}\|f\|_{\Th} \|\vr_\IT\|_{\XR}.
\]
For the second term, by continuity of $\blfdiv$ and since
$\Pi_\nabla f\in \IP^{k-1}/\IP^0$ has elementwise mean zero, we get
\[
\|\Pi_\nabla f\|_{0,h}
= \|h \nabla \Pi_\nabla f\|_{\Th}.
\]
Moreover, $\nabla\Pi_\nabla f$ is the $L^2$-orthogonal projection of $f$ onto the
space of discrete gradients, hence
\[
\|\nabla\Pi_\nabla f\|_{\Th}\le \|f\|_{\Th},
\qquad\text{and therefore}\qquad
\|\Pi_\nabla f\|_{0,h}\le h\|f\|_{\Th},
\]
and thus
\[
|\blfdiv(\vr_\IT,\Pi_\nabla f)|
\le \contdiv \|\vr_\IT\|_{1,h} \|\Pi_\nabla f\|_{0,h}
\le \contdiv h \nu^{-\frac12}\|f\|_{\Th} \|\vr_\IT\|_{\XR}.
\]
For the local term there holds
\[
|(f,h\nu^{-\frac12}\rr_h)_\Th|
\le h\nu^{-\frac12}\|f\|_{\Th} \|\rr_h\|_{\QR}.
\]

Collecting the three bounds gives
\[
\bigl|(f,\vr_\IT)_\Th - \blfdiv(\vr_\IT,\Pi_\nabla f)
+ (f,h\nu^{-\frac12}\rr_h)_\Th\bigr|
\le
\nu^{-\frac12}\|f\|_{\Th}
\Bigl(
(C_\Omega+\contdiv h)\|\vr_\IT\|_{\XR}
+
h\|\rr_h\|_{\QR}
\Bigr),
\]
thus using Cauchy--Schwarz,
\[
\bigl|(f,\vr_\IT)_\Th - \blfdiv(\vr_\IT,\Pi_\nabla f)
+ (f,h\nu^{-\frac12}\rr_h)_\Th\bigr|
\le
C_F \nu^{-\frac12}\|f\|_{\Th} \|\zr_h\|_{\ZR}.
\]
Now apply the reduced inf-sup estimate \eqref{eq:Khr_infsup} from
\cref{th:Khinfsup} to the solution $\ur_h$ of \eqref{eq:redtdgoseen}:
\[
\stabtrefftzr{w_h}\|\ur_h\|_{\XR}
\le
\sup_{\zr_h\in\ZR\setminus\{0\}}
\frac{\blftrefftzr{w_h}(\ur_h,\zr_h)}{\|\zr_h\|_{\ZR}}
=
\sup_{\zr_h\in\ZR\setminus\{0\}}
\frac{(f,\vr_\IT)_\Th - \blfdiv(\vr_\IT,\Pi_\nabla f)
+ (f,h\nu^{-\frac12}\rr_h)_\Th}{\|\zr_h\|_{\ZR}}
\]
Together with the preceding bound, this gives the claim.
\end{proof}

\subsection{Stability of the full formulation}
\label{ssec:oseen_full_stability}

We now pass from the reduced velocity formulation to the full velocity-pressure formulation.
This requires two additional ingredients.
First, the pressure modes which are not represented in the reduced velocity problem are controlled by the piecewise-constant LBB estimate, from \cref{lem:LBB_P0}.
Second, the stability of the local operator on the full local complement $\IL$ established in \cref{cor:AKw}.

The resulting argument has two levels.
We first prove an inf-sup estimate for the full Oseen form on the Trefftz space $\IT_w$.
This step combines the reduced velocity stability with the pressure control supplied by the LBB estimate.
We then obtain the stability of the full coupled Trefftz formulation on $\IX$ by adding the local residual equation and using the stability and continuity of $\oposeen{w}$ on the complement.
The final step is done by applying the framework of \cite[\S 2.5]{LLSV_ARXIV_2024}

\begin{theorem}[Inf-sup stability of the full Oseen form on $\IT_w$]
\label{thm:oseen_full_Trefftz_infsup}
Assume \eqref{ass:wh} and \eqref{ass:A2}, and let $\lambda$ be chosen as in
\cref{lem:coerc_IT_perp}. Then there exists a constant $\staboseen{w}>0$ such that
\begin{equation}\label{eq:oseen_full_Trefftz_infsup}
\sup_{\vq_\IT\in \IT_w\setminus\{0\}}
\frac{\blfoseen{w}(\up_\IT,\vq_\IT)}{\dgnorm{\vq_\IT}}
\ge
\staboseen{w}\,\dgnorm{\up_\IT}
\qquad\forall\,\up_\IT\in \IT_w,
\end{equation}
with the constant
\begin{equation}\label{eq:alpha_oseen_full}
\staboseen{w}
=
C\min\left\{
\staboseenr{w},
\frac{\alpha_0^2}{(1+\contlocal{w}^2)\contoseen{w}}
\right\},
\end{equation}
where $\staboseenr{w}$ is the coercivity constant from
\cref{lem:coerc_IT_perp}, $\alpha_0$ is the constant from
\cref{lem:LBB_P0}, 
with $C>0$ a constant that depends only on the equivalence constant in \eqref{eq:plift_normequiv}.
\end{theorem}
\begin{proof}
By \eqref{eq:trefftz_sp_split}, every $\up_\IT\in \IT_w$ can be written uniquely as
\[
\up_\IT=(u_\IT,\Plift{w} u_\IT+p_h^0),
\qquad
u_\IT\in \IT_w^u,\quad p_h^0\in \IP^0.
\]
With this splitting, the Oseen form on $\IT_w$ reads
\[
\blfoseen{w}(\up_\IT,\vq_\IT)
=
\blfoseenr{w}(u_\IT,v_\IT)
+
\blfdiv(v_\IT,p_h^0)
+
\blfdiv(u_\IT,q_h^0).
\]
Thus the full Oseen form on $\IT_w$ has the standard saddle-point structure with
velocity block $\blfoseenr{w}$ and coupling $\blfdiv(\cdot,\cdot)$ on $\IP^0$.

By \eqref{eq:plift_normequiv}, the DG norm $\dgnorm{\cdot}$ on $\IT_w$ is equivalent
to the natural product norm on $\IT_w^u\times \IP^0$.
The continuity of the full form follows from \cref{lem:blfcont}, coercivity on the
kernel is exactly \cref{lem:coerc_IT_perp}, and the inf-sup condition for the
piecewise-constant pressure part is \cref{lem:LBB_P0}. The classical
Babu\v{s}ka--Brezzi theorem therefore yields \eqref{eq:oseen_full_Trefftz_infsup}.
Inserting the corresponding continuity, coercivity, and inf-sup constants, and
using the norm equivalence \eqref{eq:plift_normequiv}, gives \eqref{eq:alpha_oseen_full}.
\end{proof}

\begin{theorem}[Inf-sup stability of the full coupled Trefftz form on $\IX$]
\label{thm:trefftz_full_infsup}
Assume \eqref{ass:wh} and \eqref{ass:A2}, and let $\lambda$ be chosen as in
\cref{lem:coerc_IT_perp}. Then there exists a constant $\stabtrefftz{w}>0$ such that
\begin{equation}\label{eq:trefftz_full_infsup}
\sup_{\z_h\in \IZ_w\setminus\{0\}}
\frac{\blftrefftz{w}(\up_h,\z_h)}{\|\z_h\|_{\IZ}}
\ge
\stabtrefftz{w}\,\dgnorm{\up_h}
\qquad\forall\,\up_h\in \IX.
\end{equation}
Moreover, one may take
\begin{equation}\label{eq:alpha_trefftz_full}
\stabtrefftz{w}
:=
C\staboseen{w}
\left(1+\frac{\contlocal{w}\contoseen{w}^2}{\stablocal{w}^2\staboseen{w}}
+\contlocal{w}\frac{\staboseen{w}}{\stablocal{w}^2}\right)^{-1}.
\end{equation}
where $\staboseen{w}$ is the constant from
\cref{thm:oseen_full_Trefftz_infsup}, and $C>0$ depends only on the norm-equivalence constants in \eqref{eq:plift_normequiv} 
and norm equicalence constant in $\dgsnorm{\vq}\lesssim\dgnorm{\vq}$ for all $\vq\in\IX_h$.
\end{theorem}
\begin{proof}
We apply the theory of coupled Trefftz problems from \cite[\S 2.5]{LLSV_ARXIV_2024} (and also \cite{SV_ARXIV_2026} for explicit constants).
We satisfy the assumptions of the theorem as follows by the results we have established so far:
lower bound on the local operator \cref{cor:AKw}, 
continuity of the local operator from \cref{lem:AKw_continuity}, 
inf-sup stability of the global bilinear form \cref{thm:oseen_full_Trefftz_infsup},
continuity of the global bilinear from \cref{lem:blfcont}. 
This proves \eqref{eq:trefftz_full_infsup}.
\end{proof}

\subsection{A priori error estimates}

To prove a C\'ea-type error estimate and a priori error estimates, we need to assume an improved consistency bound for the convection term. 

\begin{theorem}[A priori error estimate]\label{thm:oseen_apriori}
  Assume \eqref{ass:wh}, \eqref{ass:A2}, and \cref{ass:ch}.
  Let $(u,p)$ be a smooth solution to the Oseen problem with
  $u \in [H^{k+1}(\Th)]^d$ and $p \in H^\ell(\Th)$, and let $\ur_h\in\XR$
  be the solution of \eqref{eq:redtdgoseen}.
  Then for $1 \le \ell \le k$,
  \begin{equation}\label{eq:oseen_apriori}
    \|u-\ur_h\|_{\XR}
    \le
    \left(1+\frac{C_{\rm app}(w)}{\stabtrefftzr{w}}\right)
    \inf_{\substack{\vr_h\in\XR\\ \Pi^0\vr_h=\Pi^0 u}}
    \|u-\vr_h\|_{\XR,\ast}
    +
    \frac{C_p}{\stabtrefftzr{w}}\nu^{-1/2}h^\ell\|p\|_{H^\ell(\Th)},
  \end{equation}
  with
  \[
  C_{\rm app}(w) :=
  \contvis + (2\contdiv+1)\contlocal{w} + C_{t,\ast}|w|_{h,d},
  \qquad C_p := C(1+\contdiv),
  \]
  where $C>0$ depends only on $k$, $d$, and the shape-regularity.
  In particular for $k\geq2$, with the BDM interpolant $\vr_h = \Pi^{\BDM}_{k} u \in \XR$, we have that $\Pi^0(\Pi_k^{\BDM}u)=\Pi^0 u$ elementwise and therefore
  \begin{equation}\label{eq:oseen_apriori_BDM}
    \norm{u - \ur_h}_{\XR}
    \;\lesssim\;
    h^\ell \left( \nu^{\frac12}  |u|_{H^{\ell+1}(\Th)}
    \;+\;
    \nu^{-\frac12} \|p\|_{H^\ell(\Th)} \right).
  \end{equation}
\end{theorem}
\begin{proof}
Let $\vr_h \in \XR$ be such that $e_h:=u-\vr_h\in\IX_{\rm ps}(\Th)$.
By the triangle inequality and the inf-sup bound \eqref{eq:Khr_infsup} applied to $\ur_h - \vr_h$,
\[
  \norm{u - \ur_h}_{\XR}
  \le \norm{u - \vr_h}_{\XR} + \norm{\vr_h - \ur_h}_{\XR},
  \qquad
  \stabtrefftzr{w}\norm{\vr_h - \ur_h}_{\XR}
  \le \sup_{\zr_h \ne 0}
  \frac{\blftrefftzr{w}(\ur_h - \vr_h,\, \zr_h)}{\|\zr_h\|_{\ZR}}.
\]
Writing $\zr_h = (\vr_\IT, \rr_h) \in \ZR$, we split
\[
  \blftrefftzr{w}(\ur_h - \vr_h, \zr_h)
  = \underbrace{\blfoseenr{w}(\ur_h - \vr_h, \vr_\IT)}_{=:I}
  + \underbrace{\langle \oposeenr{w}(\ur_h - \vr_h), \rr_h\rangle_\Th}_{=:II}.
\]

\emph{Bound for $I$.}
For all $p_h^0 \in \IP^0$, by consistency of the DG bilinear form and \eqref{eq:redtdgoseen},
\begin{align*}
  \blfoseen{w}((u,p-p_h^0),(\vr_\IT,\Plift{w} \vr_\IT))
  &= (f,\vr_\IT)_\Th
  = \blfoseen{w}((\ur_h,\Plift{w} \ur_h + \Pi_\nabla f),(\vr_\IT,\Plift{w} \vr_\IT))
  \quad\forall \vr_\IT \in \TR_w.
\end{align*}
Setting $\tilde p := p - p_h^0 - \Pi_\nabla f$ with $p_h^0 = \Pi^0 p$ (so $\Pi^0\tilde p = 0$),
\begin{equation}\label{eq:consistency_star}\tag{$\ast$}
  \blfoseen{w}((u - \ur_h,\, \tilde p - \Plift{w} \ur_h),(\vr_\IT,\Plift{w} \vr_\IT)) = 0
  \quad\forall \vr_\IT \in \TR_w.
\end{equation}
Applying \eqref{eq:consistency_star}, we obtain
\[
  I
  = \blfoseen{w}((\ur_h-\vr_h,\Plift{w}(\ur_h-\vr_h)),(\vr_\IT,\Plift{w}\vr_\IT))
  = \blfoseen{w}((e_h,\tilde p-\Plift{w}\vr_h),(\vr_\IT,\Plift{w}\vr_\IT)).
\]
Using the decomposition of $\blfoseen{w}$, continuity of $\blfvis$, \cref{ass:ch}, continuity of $\blfdiv$, and
\[
\nu^{-\frac12}\|\Plift{w}\vr_\IT\|_{0,h}
\le
\contlocal{w}\,\|\vr_\IT\|_{\XR},
\]
we infer
\begin{align*}
|I| &\le |\blfvis(e_h,\vr_\IT)| + |\blfconv{w}(e_h,\vr_\IT)| + |\blfdiv(e_h,\Plift{w}\vr_\IT)| + |\blfdiv(\vr_\IT,\tilde p-\Plift{w}\vr_h)| \\
&\le \Bigl( \contvis + C_{t,\ast}|w|_{h,d} + \contdiv\,\contlocal{w} \Bigr)\|e_h\|_{\XR,\ast}\,\|\vr_\IT\|_{\XR} + |\blfdiv(\vr_\IT,\tilde p-\Plift{w}\vr_h)|.
\end{align*}
We bound the pressure term
\[
\nu^{-\frac12}\|\tilde p - \Plift{w}\vr_h\|_{0,h}
= \nu^{-\frac12}\|h\nabla(\tilde p - \Plift{w}\vr_h)\|_{\Th},
\]
by splitting at $\Plift{w}u$:
\[
  \nu^{-\frac12}\|\tilde p - \Plift{w}\vr_h\|_{0,h}
  \le \underbrace{\nu^{-\frac12}\|h\nabla(p - \Pi_\nabla f - \Plift{w} u)\|_\Th}_{=:I_a}
  + \underbrace{\nu^{-\frac12}\|h\nabla\Plift{w}(u-\vr_h)\|_\Th}_{=:I_b}.
\]
For $I_b$: since $\Plift{w} = -\Pi_\nabla\opcd{w}$ (see \eqref{eq:pressure_lifting}),
\[
I_b = \nu^{-\frac12}\|h\Pi_\nabla\opcd{w}(u-\vr_h)\|_\Th
\le \contlocal{w}\norm{u-\vr_h}_{\XR,\ast}.
\]
For $I_a$: using $\opcd{w}u = \Pi^{k-2}(f - \nabla p)$, \eqref{eq:pressure_lifting}, and $\Pi_\nabla\circ\Pi^{k-2} = \Pi_\nabla$, we find
\[
  \Pi_\nabla f + \Plift{w} u
  = \Pi_\nabla f - \Pi_\nabla\opcd{w}u
  = \Pi_\nabla f - \Pi_\nabla(f - \nabla p)
  = \Pi_\nabla(\nabla p).
\]
Hence
\[
  \nabla(p - \Pi_\nabla f - \Plift{w} u)
  = \nabla p - \Pi_\nabla(\nabla p),
\]
which gives
\[
I_a \lesssim \nu^{-\frac12}\|h(\nabla p - \Pi_\nabla \nabla p)\|_\Th
\lesssim \nu^{-\frac12}h^\ell\|p\|_{H^\ell(\Th)},
\]
for $1\le\ell\le k$ by a standard Bramble--Hilbert argument.
Therefore,
\begin{align*}
|\blfdiv(\vr_\IT,\tilde p-\Plift{w}\vr_h)|
&\le \contdiv \|\vr_\IT\|_{\XR}\,\nu^{-\frac12}\|\tilde p-\Plift{w}\vr_h\|_{0,h} \\
&\le \contdiv\Big( \contlocal{w}\|e_h\|_{\XR,\ast} + C\,\nu^{-\frac12}h^\ell\|p\|_{H^\ell(\Th)} \Big)\|\vr_\IT\|_{\XR},
\end{align*}
and hence,
\begin{equation}\label{eq:I_bound_apriori}
|I| \le
\Big( \bigl[\contvis + 2\contdiv\,\contlocal{w} + C_{t,\ast}|w|_{h,d}\bigr]\|e_h\|_{\XR,\ast} + C\,\contdiv\,\nu^{-\frac12}h^\ell\|p\|_{H^\ell(\Th)} \Big)\|\zr_h\|_{\ZR}.
\end{equation}

\emph{Bound for $II$.}
The local equation \eqref{eq:tdgoseen:2} and the PDE itself yield for $\rr_h \in \QR$:
\begin{align*}
  &\langle \oposeenr{w}u, r_h \rangle_\Th 
   =
  \langle h \nu^{-\frac12} (-\nu \Delta u + w\cdot\nabla u), r_h \rangle_\Th 
  =
  \langle -h \nu^{-\frac12} \nabla p, r_h \rangle_\Th + \langle h \nu^{-\frac12} f, r_h \rangle_\Th \\
  \stackrel{\eqref{eq:tdgoseen:2}}{\Longrightarrow} \
  &\langle \oposeenr{w}(u-\ur_h), r_h \rangle_\Th 
   =
  \langle -h \nu^{-\frac12} \nabla p, r_h \rangle_\Th
  =
  \langle h \nu^{-\frac12} \nabla (q_h-p), r_h \rangle_\Th,
\end{align*}
for any $q_h \in \IP^{k-1}$ where in the last step we exploited that $\QR$ is orthogonal to discrete gradients ($\rr_h\perp\nabla\IP^{k-1}(T)$).
Applying that result and continuity from \eqref{eq:AKrw_cont_coer} to $II$, we obtain:
\begin{align*}
  II & = \langle\oposeenr{w}(\ur_h-\vr_h),\rr_h\rangle_\Th
  = \langle\oposeenr{w}(u-\vr_h),\rr_h\rangle_\Th + \langle h \nu^{-\frac12} \nabla (q_h-p), r_h \rangle_\Th \\
  & \le  \|\oposeenr{w}(u-\vr_h)\|_{\QR'} \,\|\rr_h\|_{\QR} + \nu^{-\frac12} \inf_{q_h\in\IP^{k-1}} \norm{h_{\calT}(p-q_h)}_{H^1(\Th)} \,\|\rr_h\|_{\QR} \\
  & \le \contlocal{w}\norm{u-\vr_h}_{\XR,\ast}\,\|\zr_h\|_{\ZR} + C \nu^{-\frac12} h^\ell \Vert p \Vert_{H^\ell(\Th)}\,\|\zr_h\|_{\ZR}.
\end{align*}

\emph{Assembly.}
Combining \eqref{eq:I_bound_apriori} with the estimate for $II$, dividing by $\|\zr_h\|_{\ZR}$, yields
\[
  \stabtrefftzr{w}\norm{\vr_h - \ur_h}_{\XR}
  \le C_{\rm app}(w)\norm{u-\vr_h}_{\XR,\ast}
  + C_p\,\nu^{-\frac12}h^\ell\|p\|_{H^\ell(\Th)}.
\]
The triangle inequality gives
\[
\norm{u-\ur_h}_{\XR}
\le
\left(
1+\frac{C_{\rm app}(w)}{\stabtrefftzr{w}}
\right)\norm{u-\vr_h}_{\XR,\ast}
+
\frac{C_p}{\stabtrefftzr{w}}\,
\nu^{-\frac12}h^\ell\|p\|_{H^\ell(\Th)}.
\]
Taking the infimum over all $\vr_h\in\XR$ with $u-\vr_h\in\IX_{\rm ps}(\Th)$ proves \eqref{eq:oseen_apriori}.

For \eqref{eq:oseen_apriori_BDM}, the standard BDM estimate gives 
$\Pi^{\BDM}_k u \in  [\IP^k]^d$ with
$\Pi^0(\Pi_k^{\BDM}u)=\Pi^0 u$ elementwise for $k\ge 2$,
$\Div \Pi^{\BDM}_k u|_\Th = 0$ (for $\Div u = 0$) and $\jmp{\Pi^{\BDM}_k u \cdot n} = 0$ on all $F \in \Fh$. Further,
$\norm{u-\Pi^{\BDM}_k u}_{\XR,\ast} = \nu^{\frac12}\vdgsnorm{u-\Pi^{\BDM}_k u} \lesssim \nu^{\frac12}h^k|u|_{H^{k+1}(\Th)}.$
Inserting this choice into \eqref{eq:oseen_apriori} yields \eqref{eq:oseen_apriori_BDM}.
\end{proof}

\begin{lemma}[Pressure error bound]\label{lem:pressure_error}
Assume \eqref{ass:wh} and \eqref{ass:A2}, and in addition
$w\in [W^{\ell_2,\infty}(\Th)]^d$ for 
$1\le\ell_2\le k$.
Let $(u,p)$ be a Oseen solution with $p\in H^\ell(\Th)$ and let
$(u_h,p_h)\in\IX$ be the solution of \eqref{eq:tdgoseen}.
Then for $0\le\ell_1\le k-1$,
\begin{align}\label{eq:pressure_error}
  \nu^{-\frac12}\|p - p_h\|_{0,h}
  \;\lesssim\;
   \norm{u - u_h}_{\XR}
  \;&+\;
  \nu^{-\frac12} h^{\ell_1+1}\|f\|_{H^{\ell_1}(\Th)}
  \\ &+\; \nonumber
  \left( \nu^{\frac12} + \nu^{-\frac12}|w|_{W^{\ell_2,\infty}(\Th)} h\right)
  h^{\ell_2}\|u\|_{H^{\ell_2+1}(\Th)},
\end{align}
where the hidden constant depends on $\contvis$, $\contdiv$, $\contlocal{w}$, $\alpha_0$, and 
on $C_{t,\ast}|w|_{h,d}$.
\end{lemma}
\begin{proof}
Split $p - p_h = (\tilde p - \tilde p_h) + (p^0 - p_h^0)$
with $p^0 := \Pi^0 p$, $\tilde p := p - p^0$, and 
$p_h^0 = \Pi^0 p_h$. 
With the local problem, we find $\tilde p_h = \Plift{w} u_h + \Pi_\nabla f$.
From the Trefftz property and the PDE we get
\begin{align*}
\nabla p_h & = \Pi^{k-2} f - \Pi^{k-2} \opcd{w} u_h, \quad \nabla p = f + \nu \Delta u - w \cdot \nabla u,
\end{align*}
which implies
\begin{align*}
& \Vert h \nabla (p - p_h) \Vert_\Th \\
& \leq \norm{ h (\id - \Pi^{k-2})f}_\Th + \norm{ h (\id - \Pi^{k-2})(-\nu \Delta u + w\cdot \nabla u)}_\Th
 + \norm{ h \opcd{w}(u-u_h)}_\Th \\
 &
 \leq C h^{\ell_1+1} \Vert f \Vert_{H^{\ell_1}(\Th)}
 + C \nu  h^{\ell_2} \Vert u \Vert_{H^{\ell_2+1}(\Th)}
 + C |w|_{W^{\ell_2-1,\infty}(\Th)} h^{\ell_2} \Vert u \Vert_{H^{\ell_2}(\Th)}
 + \contlocal{w} \nu^{\frac12} \norm{ u-u_h}_{\XR},
\end{align*}
where we used the local approximation property of $\Pi^{k-2}$ and
\[
\|(\id-\Pi^{k-2})(w\cdot\nabla u)\|_{T}
\lesssim
h_T^{\ell_2-1}\,|w\cdot\nabla u|_{H^{\ell_2-1}(T)}
\lesssim
h_T^{\ell_2-1}\,|w|_{W^{\ell_2-1,\infty}(T)}\,\|u\|_{H^{\ell_2}(T)},
\]
for $0 \le \ell_1 \le k-1$, $1 \le \ell_2 \le k$. 

It remains to bound the $\IP^0$-contribution.
By \eqref{eq:lbb_P0} we get
\[
\alpha_0 \|p_h^0-p^0\|_{0,h}
\le
\sup_{v_\IT\in\IT_w^u\setminus\{0\}}
\frac{\blfdiv(v_\IT,p_h^0-p^0)}{\|v_\IT\|_{1,h}}.
\]
Moreover we have the splitting
\begin{align*}
\blfdiv(v_\IT,p_h^0-p^0)
&=
\blfdiv(v_\IT,p_h-p)+\blfdiv(v_\IT,\tilde p-\tilde p_h).
\end{align*}
The second term is bounded by
\[
|\blfdiv(v_\IT,\tilde p-\tilde p_h)|
\le
\contdiv \|v_\IT\|_{1,h}\,\nu^{-\frac12}\|\tilde p-\tilde p_h\|_{0,h}.
\]
For the first term, using the error identity for Trefftz test functions, together with the continuity of $\blfvis$ and, if desired, \cref{ass:ch} for the convection part, we obtain
\[
|\blfdiv(v_\IT,p_h-p)|
\le
C\Big(\|u-u_h\|_{\XR,\ast}
+\nu^{-\frac12}\|\tilde p-\tilde p_h\|_{0,h}\Big)\|v_\IT\|_{1,h},
\]
with $C$ depending on $\contvis$, $\contdiv$, $\contlocal{w}$, and, under \cref{ass:ch}, on $C_{t,\ast}|w|_{h,d}$.
Therefore,
\[
\|p_h^0-p^0\|_{0,h}
\le
C\alpha_0^{-1}
\Big(
\|u-u_h\|_{\XR,\ast}
+
\nu^{-\frac12}\|\tilde p-\tilde p_h\|_{0,h}
\Big).
\]
\end{proof}

\section{Numerical results}\label{sec:numerics}

In the numerical examples we choose the discrete trilinear form $\blfconv{w}$ as a discrete counterpart of Temam's modification given in \eqref{eq:ch_temam}.
For the implementation of the methods we use \texttt{NGSolve} \cite{ngsolve} and \texttt{NGSTrefftz} \cite{ngstrefftz}.
Replication data are available in \cite{stocker_2026_20490547}.

\subsection{Prescribed Oseen--Kovasznay flow}\label{sec:kovasznay_oseen}

As a first test we consider the Kovasznay flow, but now interpreted as an
\emph{exact solution of the stationary Oseen problem with prescribed convection
field}.
More precisely, we set
\begin{eqs}
    u = \begin{pmatrix}
    1-\exp(\kappa x)\cos(2\pi y)\\[0.2em]
    \kappa/(2\pi)\exp(\kappa x)\sin(2\pi y)
    \end{pmatrix},
    \qquad
    p = -\frac12\exp(2\kappa x)+\bar p,
\end{eqs}
with
\[
\kappa = \frac{1}{2\nu}-\sqrt{\frac{1}{4\nu^2}+4\pi^2},
\]
where $\bar p$ is chosen such that $\int_\Omega p = 0$.
We then prescribe the convecting field as
\[
w=u,
\]
so that $(u,p)$ solves the Oseen problem with right-hand side $f=0$.
As in the Navier--Stokes experiment, we impose inhomogeneous Dirichlet boundary
conditions matching the exact velocity.
The penalty parameter is chosen as $\lambda=50 k^2$.

We choose the computational domain
\[
\Omega = [-0.5,1.5]\times [0,2],
\]
set $\nu=1.0$, and compare the Trefftz-DG method with $\IT_w$ trial space
against the standard DG method with $\IP^k$ elements for $k=2,3,4,5$.

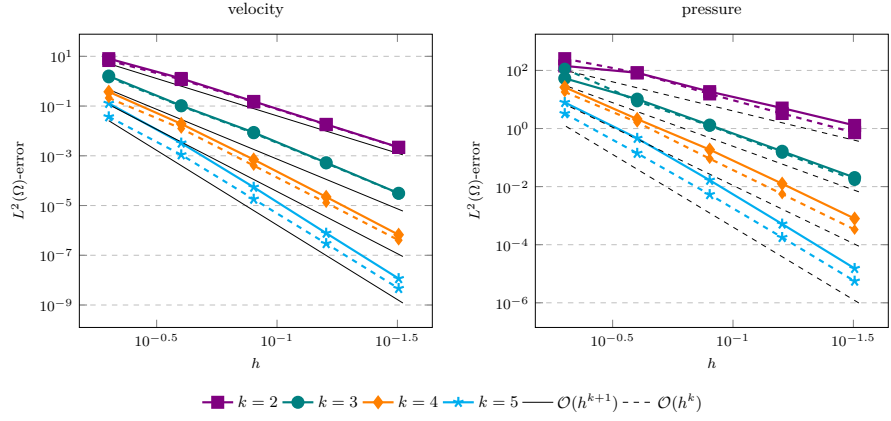
\begin{figure}[ht!]
\begin{center}
\resizebox{.8\linewidth}{!}{
\begin{tikzpicture}
\begin{groupplot}[%
  group style={
    group name={my plots},
    group size=2 by 1,
    horizontal sep=2cm,
  },
  legend style={
    legend columns=8,
    at={(0.5,-0.2)},
    draw=none
  },
  ymajorgrids=true,
  grid style=dashed,
  cycle list name=paulcolors4,
]
\nextgroupplot[
    title={velocity},
    ymode=log,
    xmode=log,
    x dir=reverse,
    ylabel={$L^2(\Omega)$-error},
    xlabel={$h$}
]
\foreach \p in {2,3,4,5}{
    \addplot+[discard if not={trefftz}{1},discard if not={p}{\p},discard if not={nu}{1.0}]
    table [x=h, y=ul2error, col sep=comma] {oseen_kovasznay.csv};
}
\foreach \p in {2,3,4,5}{
    \addplot+[discard if not={trefftz}{0},discard if not={p}{\p},discard if not={nu}{1.0}, dashed]
    table [x=h, y=ul2error, col sep=comma] {oseen_kovasznay.csv};
}
\addplot[black,domain=0.03:0.5] {exp(-3*ln(1/x)+3.7)};
\addplot[black,domain=0.03:0.5] {exp(-4*ln(1/x)+2)};
\addplot[black,domain=0.03:0.5] {exp(-5*ln(1/x)+1.3)};
\addplot[black,domain=0.03:0.5] {exp(-6*ln(1/x)+0.5)};

\nextgroupplot[
    title={pressure},
    ymode=log,
    xmode=log,
    x dir=reverse,
    ylabel={$L^2(\Omega)$-error},
    xlabel={$h$}
]
\foreach \p in {2,3,4,5}{
    \addplot+[discard if not={trefftz}{1},discard if not={p}{\p},discard if not={nu}{1.0}]
    table [x=h, y=pl2error, col sep=comma] {oseen_kovasznay.csv};
}
\addlegendimage{solid}
\addlegendimage{dashed}
\foreach \p in {2,3,4,5}{
    \addplot+[discard if not={trefftz}{0},discard if not={p}{\p},discard if not={nu}{1.0}, dashed]
    table [x=h, y=pl2error, col sep=comma] {oseen_kovasznay.csv};
}
\addplot[black,dashed,domain=0.03:0.5] {exp(-2*ln(1/x)+6)};
\addplot[black,dashed,domain=0.03:0.5] {exp(-3*ln(1/x)+5.5)};
\addplot[black,dashed,domain=0.03:0.5] {exp(-4*ln(1/x)+4.7)};
\addplot[black,dashed,domain=0.03:0.5] {exp(-5*ln(1/x)+3.7)};
\legend{$k=2$,$k=3$,$k=4$,$k=5$,$\mathcal O(h^{k+1})$,$\mathcal O(h^k)$}
\end{groupplot}
\end{tikzpicture}}
\end{center}
\vspace{-1.5em}
\caption{
Convergence of the velocity and pressure in the $L^2$-norm for the prescribed
Oseen--Kovasznay problem.
The Trefftz-DG results are shown in solid lines, while the standard DG results
are shown in dashed lines.
The expected convergence rates are indicated by the black reference lines.
}
\label{fig:oseen_kovasznay}
\end{figure}

The results in \cref{fig:oseen_kovasznay} show that both discretizations recover
the expected asymptotic convergence rates,
namely $\mathcal O(h^{k+1})$ for the velocity and $\mathcal O(h^{k})$ for the
pressure in the $L^2$-norm.
Compared at the same polynomial degree, the standard DG method yields slightly
smaller error constants, whereas the Trefftz-DG method exhibits the same
asymptotic slopes.

\begin{figure}[ht!]
\begin{center}
\resizebox{.8\linewidth}{!}{
\begin{tikzpicture}
\begin{groupplot}[%
  group style={
    group name={my plots},
    group size=2 by 1,
    horizontal sep=2cm,
  },
  legend style={
    legend columns=8,
    at={(0.3,-0.2)},
    draw=none
  },
  ymajorgrids=true,
  grid style=dashed,
  cycle list name=paulcolors4,
]
\nextgroupplot[
    title={velocity},
    ymode=log,
    xmode=log,
    ylabel={$L^2(\Omega)$-error},
    xlabel={DoFs}
]
\foreach \p in {2,3,4,5}{
    \addplot+[discard if not={trefftz}{1},discard if not={p}{\p},discard if not={nu}{1.0}]
    table [x=ndof, y=ul2error, col sep=comma] {oseen_kovasznay.csv};
}
\foreach \p in {2,3,4,5}{
    \addplot+[discard if not={trefftz}{0},discard if not={p}{\p},discard if not={nu}{1.0}, dashed]
    table [x=ndof, y=ul2error, col sep=comma] {oseen_kovasznay.csv};
}

\nextgroupplot[
    title={pressure},
    ymode=log,
    xmode=log,
    ylabel={$L^2(\Omega)$-error},
    xlabel={DoFs}
]
\foreach \p in {2,3,4,5}{
    \addplot+[discard if not={trefftz}{1},discard if not={p}{\p},discard if not={nu}{1.0}]
    table [x=ndof, y=pl2error, col sep=comma] {oseen_kovasznay.csv};
}
\foreach \p in {2,3,4,5}{
    \addplot+[discard if not={trefftz}{0},discard if not={p}{\p},discard if not={nu}{1.0}, dashed]
    table [x=ndof, y=pl2error, col sep=comma] {oseen_kovasznay.csv};
}
\legend{$k=2$,$k=3$,$k=4$,$k=5$}
\end{groupplot}
\end{tikzpicture}}
\end{center}
\vspace{-1.5em}
\caption{
 As in \cref{fig:oseen_kovasznay}, but with the error plotted against the number of degrees of freedom (DoFs) instead of the mesh size $h$.
}
\label{fig:oseen_kovasznay_ndof}
\end{figure}
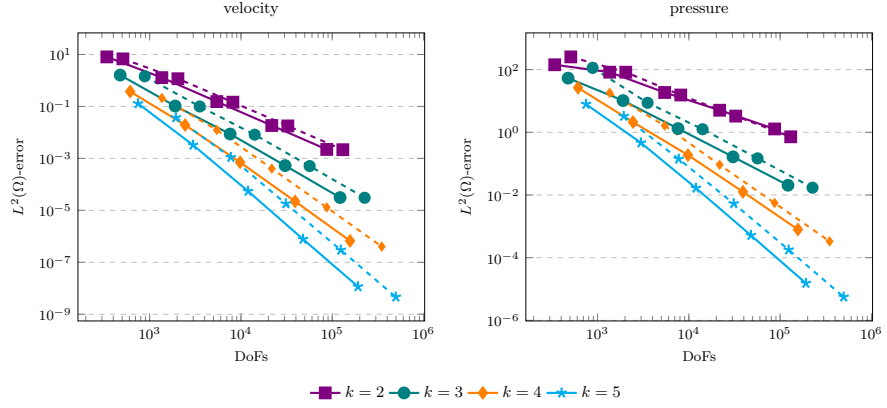

In \cref{fig:oseen_kovasznay_ndof}, we plot the same errors against the number of degrees of freedom (DoFs) instead of the mesh size $h$.
This comparison shows that the Trefftz-DG method achieves a given error level with significantly fewer DoFs than the standard DG method, especially for higher polynomial degrees.

\begin{table}[ht!]\centering
{\setlength{\tabcolsep}{4pt}\begin{filecontents*}{oseen_kovasznay.csv}
hnr,nu,h,p,alpha,trefftz,ul2error,ul2eoc,pl2error,pl2eoc,totaltime,ndof
0,1.0,0.5,2,50,1,8.066459522850618,0.0,142.95834175464924,0.0,0.028371334075927734,340
1,1.0,0.25,2,50,1,1.287920197186935,2.6468923910797684,82.85856003800403,0.7868721513539959,0.03882431983947754,1360
2,1.0,0.125,2,50,1,0.15361001762913862,3.0676989940211916,18.64357296169906,2.151972375704612,0.11148667335510254,5440
3,1.0,0.0625,2,50,1,0.018382037534571804,3.0629037096534266,5.065172309866961,1.8799951194947282,0.4293692111968994,21760
4,1.0,0.03125,2,50,1,0.0022072545881353595,3.0579717422757593,1.2797044218359366,1.984800724289261,1.5407609939575195,87040
0,1.0,0.5,3,50,1,1.6106074796734828,0.0,53.49602814194146,0.0,0.020848751068115234,476
1,1.0,0.25,3,50,1,0.10369355966772806,3.9572067411127847,10.288683626977257,2.378373371181597,0.050473690032958984,1904
2,1.0,0.125,3,50,1,0.008670035401634073,3.5801445976632293,1.3163633865732085,2.9664287001684024,0.1820850372314453,7616
3,1.0,0.0625,3,50,1,0.0005271477968622295,4.039758471008933,0.16486782804731798,2.9971759979572004,0.7207260131835938,30464
4,1.0,0.03125,3,50,1,3.110856842774188e-05,4.082823597444089,0.02047755656604218,3.0091944171476754,2.744845390319824,121856
0,1.0,0.5,4,50,1,0.37505056839714696,0.0,26.385347917014386,0.0,0.025974512100219727,612
1,1.0,0.25,4,50,1,0.019207618995565556,4.28733453194073,2.1184350718766956,3.6386661898323385,0.0785369873046875,2448
2,1.0,0.125,4,50,1,0.0007068803986285416,4.764068743902727,0.18783425913406807,3.4954667864915745,0.28577613830566406,9792
3,1.0,0.0625,4,50,1,2.2192313725984993e-05,4.993334144408259,0.012595223766110412,3.898511562214222,1.0000121593475342,39168
4,1.0,0.03125,4,50,1,6.666959211594734e-07,5.0568873768889935,0.0007921371133189608,3.9909827707176304,5.217082977294922,156672
0,1.0,0.5,5,50,1,0.12847911472254858,0.0,7.90049859813494,0.0,0.06591510772705078,748
1,1.0,0.25,5,50,1,0.0032606520259732485,5.300229560868744,0.46124128593397057,4.0983501439974725,0.14700078964233398,2992
2,1.0,0.125,5,50,1,5.460690542231766e-05,5.899933274395472,0.016628444516854704,4.793796529826394,0.45278072357177734,11968
3,1.0,0.0625,5,50,1,7.759515187932383e-07,6.136973075164504,0.0005159844607929708,5.010181791341769,1.729691743850708,47872
4,1.0,0.03125,5,50,1,1.1569455742304899e-08,6.067573613119284,1.5285658942851564e-05,5.077076967411148,7.841163873672485,191488
0,1.0,0.5,2,50,0,6.777091227325906,0.0,254.6714055811564,0.0,0.020650386810302734,510
1,1.0,0.25,2,50,0,1.16620120614397,2.5388494719103103,82.7307248218512,1.6221418563114454,0.040288448333740234,2040
2,1.0,0.125,2,50,0,0.14581811320347354,2.9995748756866885,15.495538563872103,2.4165703237811176,0.21088910102844238,8160
3,1.0,0.0625,2,50,0,0.017768492965506796,3.036776709675543,3.2891690554313278,2.2360578319213187,0.6182131767272949,32640
4,1.0,0.03125,2,50,0,0.002160009613159518,3.04021168614029,0.7189464146362614,2.1937670100246978,2.6523051261901855,130560
0,1.0,0.5,3,50,0,1.4463872795136459,0.0,113.59807416072799,0.0,0.030941247940063477,884
1,1.0,0.25,3,50,0,0.09958206282897482,3.860424183735571,8.730896604424196,3.701664750123682,0.10805821418762207,3536
2,1.0,0.125,3,50,0,0.00817800939492443,3.6060642755955827,1.2527072944813196,2.8010804602561117,0.42621850967407227,14144
3,1.0,0.0625,3,50,0,0.0005056260570916279,4.0156070013964325,0.14731632436495215,3.088060144532237,1.626354455947876,56576
4,1.0,0.03125,3,50,0,3.052124291255803e-05,4.0501851945673994,0.017101147774531266,3.1067522440835535,7.951674699783325,226304
0,1.0,0.5,4,50,0,0.2119090293149604,0.0,18.121382200813205,0.0,0.05588722229003906,1360
1,1.0,0.25,4,50,0,0.012499049766978567,4.083554737264901,1.6361115550660774,3.4693499764537505,0.2023169994354248,5440
2,1.0,0.125,4,50,0,0.00040411818449676717,4.950897337449925,0.09152289462163857,4.15999462736709,0.80391526222229,21760
3,1.0,0.0625,4,50,0,1.2960768985497972e-05,4.962554047786173,0.005556245808509298,4.041950350657367,3.5459141731262207,87040
4,1.0,0.03125,4,50,0,4.0227565844638356e-07,5.009823064548923,0.00033231486272113325,4.063487701754344,15.878551483154297,348160
0,1.0,0.5,5,50,0,0.037076343098810395,0.0,3.269377154146953,0.0,0.09347820281982422,1938
1,1.0,0.25,5,50,0,0.0010989032474715927,5.076362680134611,0.1406862105853405,4.538462981184488,0.41804933547973633,7752
2,1.0,0.125,5,50,0,1.805729145176416e-05,5.927339051679525,0.005336332582432746,4.720488534875192,1.8477001190185547,31008
3,1.0,0.0625,5,50,0,2.9622697763710016e-07,5.929734664152753,0.00017807236469476753,4.905313039063005,8.184980154037476,124032
4,1.0,0.03125,5,50,0,4.635225767546885e-09,5.9979196140345765,5.586621240164257e-06,4.994343817414067,34.05643701553345,496128
\end{filecontents*}
\pgfplotstabletypeset[
col sep=comma,
fixed,
fixed zerofill=true,
every head row/.style={before row=\toprule,after row=\midrule},
every last row/.style={after row=\bottomrule},
create on use/nhnr/.style={create col/set list={1.0,$2^{-1}$,$2^{-2}$,$2^{-3}$,$2^{-4}$}},
columns={nhnr,totaltime,totaltime,totaltime,totaltime,totaltime,totaltime},
display columns/0/.style={ column name={Meshsize}, string type, column type=l, discard if not={trefftz}{1}, discard if not={p}{2}, discard if not={nu}{1.0} },
display columns/1/.style={ column name={$\IT^3$}, discard if not={trefftz}{1}, discard if not={p}{3}, discard if not={nu}{1.0}, column type/.add={r@{\hspace{10pt}}}{} },
display columns/2/.style={ column name={$\IP^3$}, discard if not={trefftz}{0}, discard if not={p}{3}, discard if not={nu}{1.0}, column type/.add={r@{\hspace{10pt}}}{} },
display columns/3/.style={ column name={$\IT^4$}, discard if not={trefftz}{1}, discard if not={p}{4}, discard if not={nu}{1.0}, column type/.add={r@{\hspace{10pt}}}{} },
display columns/4/.style={ column name={$\IP^4$}, discard if not={trefftz}{0}, discard if not={p}{4}, discard if not={nu}{1.0}, column type/.add={r@{\hspace{10pt}}}{} },
display columns/5/.style={ column name={$\IT^5$}, discard if not={trefftz}{1}, discard if not={p}{5}, discard if not={nu}{1.0}, column type/.add={r@{\hspace{10pt}}}{} },
display columns/6/.style={ column name={$\IP^5$}, discard if not={trefftz}{0}, discard if not={p}{5}, discard if not={nu}{1.0}, column type=r },
]{oseen_kovasznay.csv}}
\caption{Total timings in seconds for the prescribed Oseen--Kovasznay problem.}
\label{tab:oseen_kovasznay}
\end{table}

The timing results in \cref{tab:oseen_kovasznay} show a clear advantage of the Trefftz-DG discretization.
While the standard DG method is slightly more accurate at fixed polynomial degree, the Trefftz-DG method is consistently faster, and this advantage becomes more pronounced on finer meshes and for higher orders.
Thus, for this smooth Oseen test, the Trefftz reduction pays off primarily in terms of computational cost rather than improved error constants.

\subsection{Sensitivity with respect to viscosity and penalty parameter}\label{sec:oseen_reynolds}

We next consider the same Kovasznay Oseen problem as in \cref{sec:kovasznay_oseen}, this time we vary both the viscosity $\nu$ and the penalty parameter $\lambda$.
More precisely, we fix the polynomial degree to $k=4$, choose
$ \nu = 0.25^i,\ i=1,\dots,7,$
for $\lambda\in\{32, \,256\}$,
and consider the three mesh sizes $ h\in\{0.25,\,0.125,\,0.0625\}. $

We compare the two discretizations directly through their $L^2$ velocity errors.
In \cref{fig:oseen_errors_viscosity_penalty} we show the errors as functions of the viscosity $\nu$ for the two representative penalty parameters $\lambda=32$ and $\lambda=256$.
Solid lines correspond to the Trefftz-DG method, whereas dashed lines show the standard DG method.

\begin{figure}[ht!]
\centering
\resizebox{\textwidth}{!}{%
\begin{tikzpicture}
\begin{groupplot}[%
    group style={
        group size=3 by 1,
        horizontal sep=1.4em,
        yticklabels at=edge left,
        x descriptions at=edge bottom,
        y descriptions at=edge left,
    },
    xlabel={Viscosity $\nu$},
    ylabel={Velocity error},
    xmode=log,
    ymode=log,
    legend style={
        at={(1.5,-0.3)},
        anchor=south,
        legend columns=4,
        draw=none,
    },
    cycle list name=paulcolors2,
    ymin=1e-12,ymax=1e4
]

\nextgroupplot[title={$h=0.25$}]
\addplot+[ discard if not={h}{0.25}, discard if not={alpha}{32}, discard if not={trefftz}{1}] table[ col sep=comma, x=nu, y=ul2error ] {oseen_kovasznay_nuo4.csv};
\addplot+[ discard if not={h}{0.25}, discard if not={alpha}{32}, discard if not={trefftz}{0} ] table[ col sep=comma, x=nu, y=ul2error ] {oseen_kovasznay_nuo4.csv};
\addplot+[ discard if not={h}{0.25}, discard if not={alpha}{256},  discard if not={trefftz}{1}] table[ col sep=comma, x=nu, y=ul2error ] {oseen_kovasznay_nuo4.csv};
\addplot+[ discard if not={h}{0.25}, discard if not={alpha}{256},  discard if not={trefftz}{0}] table[ col sep=comma, x=nu, y=ul2error ] {oseen_kovasznay_nuo4.csv};
\addlegendentry{TDG, $\lambda=16$}
\addlegendentry{DG, $\lambda=16$}
\addlegendentry{TDG, $\lambda=256$}
\addlegendentry{DG, $\lambda=256$}

\nextgroupplot[title={$h=0.125$}]
\addplot+[ discard if not={h}{0.125}, discard if not={alpha}{32}, discard if not={trefftz}{1}] table[ col sep=comma, x=nu, y=ul2error ] {oseen_kovasznay_nuo4.csv};
\addplot+[ discard if not={h}{0.125}, discard if not={alpha}{32}, discard if not={trefftz}{0} ] table[ col sep=comma, x=nu, y=ul2error ] {oseen_kovasznay_nuo4.csv};
\addplot+[ discard if not={h}{0.125}, discard if not={alpha}{256},  discard if not={trefftz}{1}] table[ col sep=comma, x=nu, y=ul2error ] {oseen_kovasznay_nuo4.csv};
\addplot+[ discard if not={h}{0.125}, discard if not={alpha}{256},  discard if not={trefftz}{0}] table[ col sep=comma, x=nu, y=ul2error ] {oseen_kovasznay_nuo4.csv};

\nextgroupplot[title={$h=0.0625$}]
\addplot+[ discard if not={h}{0.0625}, discard if not={alpha}{32}, discard if not={trefftz}{1}] table[ col sep=comma, x=nu, y=ul2error ] {oseen_kovasznay_nuo4.csv};
\addplot+[ discard if not={h}{0.0625}, discard if not={alpha}{32}, discard if not={trefftz}{0} ] table[ col sep=comma, x=nu, y=ul2error ] {oseen_kovasznay_nuo4.csv};
\addplot+[ discard if not={h}{0.0625}, discard if not={alpha}{256},  discard if not={trefftz}{1}] table[ col sep=comma, x=nu, y=ul2error ] {oseen_kovasznay_nuo4.csv};
\addplot+[ discard if not={h}{0.0625}, discard if not={alpha}{256},  discard if not={trefftz}{0}] table[ col sep=comma, x=nu, y=ul2error ] {oseen_kovasznay_nuo4.csv};

\end{groupplot}
\end{tikzpicture}}
\caption{
Velocity errors of the Trefftz-DG and standard DG discretizations as functions of the viscosity~$\nu$ for two representative penalty parameters and three mesh sizes~$h$.
 Solid lines show Trefftz-DG errors, whereas dashed lines show standard DG errors.
}
\label{fig:oseen_errors_viscosity_penalty}
\end{figure}
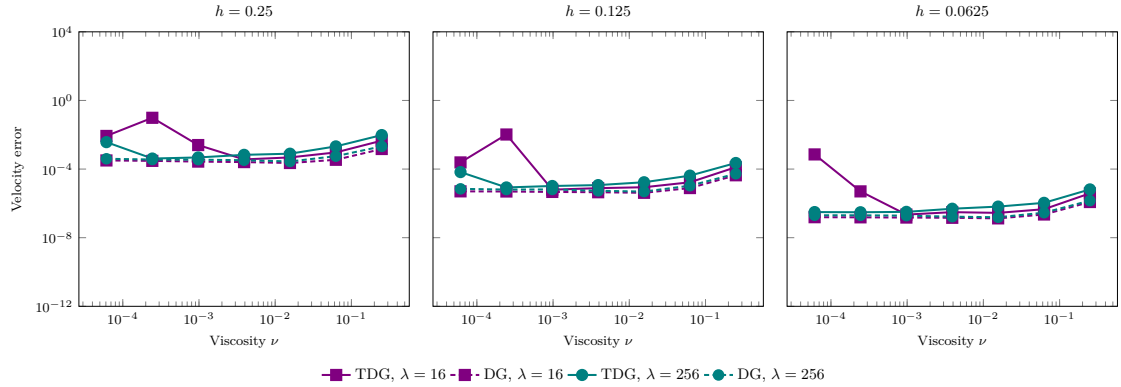

For moderate viscosities, both discretizations show comparable qualitative behaviour, although the standard DG method gives smaller errors at fixed polynomial degree, as already observed in \cref{fig:oseen_kovasznay}.
For smaller viscosities the Trefftz-DG errors deteriorate earlier than the corresponding standard DG errors.
This loss of accuracy is most pronounced for the smaller penalty parameter, while increasing the penalty from $\lambda=32$ to $\lambda=256$ delays the onset of the deterioration.
Mesh refinement has a similar stabilizing effect, indicating that the loss of accuracy is tied to the resolution and conditioning of the local embedded Trefftz construction in the convection-dominated regime.

Overall, the convection-dominated regime is clearly the most problematic for our approach, and improvements may be possible both at the level of the method itself and in the robustness of its implementation, for instance through more stable scaling strategies for the local Trefftz spaces.
A thorough treatment of Trefftz-DG methods for singularly perturbed problems is beyond the scope of this work and remains an interesting direction for future research.

\section*{Acknowledgements}
\sloppy{
This research was funded in part by the Austrian Science Fund (FWF) \href{https://doi.org/10.55776/ESP4389824}{10.55776/ESP4389824}.
For open access purposes, the author has applied a CC BY public copyright license to any author-accepted manuscript version arising from this submission.}

\appendix

\section{Auxiliary results}\label{sec:appendix}
\subsection{Discrete inequalities}\label{sec:appendix:ineq}
We collect the standard discrete inequalities used throughout the analysis.
All constants below are independent of the mesh size; their precise dependence is indicated in each case.

\smallskip\noindent
\emph{Discrete Sobolev--Poincar\'e.}
There exists $C_{\Omega}>0$, depending only on $\Omega$ and $k$, such that
\begin{equation}\label{sp}
    \|v_h\|_{\Th} \le C_{\Omega} \|v_h\|_{1,h} \quad \forall v_h\in\IP^k.
\end{equation}

\smallskip\noindent
\emph{Inverse inequality for the Laplacian.}
There exists $C_\Delta>0$, depending only on shape-regularity and $k$, such that
\begin{equation}\label{la}
    \|h_T \Delta v_h\|_{T} \le C_\Delta \|\nabla v_h\|_{T} \quad \forall v_h\in\IP^k.
\end{equation}

\smallskip\noindent
\emph{Local $L^2$--$L^4$ embedding.}
There exists $C_{4}>0$, depending only on shape-regularity and $k$, such that
\begin{equation}\label{l4}
    \|v_h\|_{L^4(T)} \le C_{4} h_T^{-\frac d4} \|v_h\|_{T} \quad \forall v_h\in\IP^k.
\end{equation}
We will often use the dual version of this estimate: for any $v\in L^{\frac43}(T)$,
\begin{equation}\label{l4dual}
   \|\Pi^{k}_T v \|_T= \|v\|_{(\IP^k(T))'} \le C_{4} h_T^{-\frac d4} \|v\|_{L^{\frac43}(T)},
\end{equation}
where the dual norm is defined with respect to the standard $L^2$-scalar product.
Indeed, by the definition of the dual norm and \eqref{l4},
\[
\|v\|_{(\IP^k(T))'}
= \sup_{r\in \IP^k(T)\setminus\{0\}} \frac{|(v,r)_T|}{\|r\|_T}
\le C_{4} h_T^{-\frac d4}
\sup_{r\in \IP^k(T)\setminus\{0\}} \frac{|(v,r)_T|}{\|r\|_{L^4(T)}}
\le C_{4} h_T^{-\frac d4} \|v\|_{L^{\frac43}(T)}.
\]

\smallskip\noindent
\emph{Discrete Sobolev embedding.}
By \cite[Theorem 5.3]{DiPietroErn}, for $q\leq 6$ there exists $C_{S,q}>0$, depending only on shape-regularity, $\Omega$, $k$ and $q$, such that
\begin{equation}\label{se}
    \|v_h\|_{L^{q}(\Omega)} \le C_{S,q} \|v_h\|_{1,h}, \quad \forall v_h\in\IP^k.
\end{equation}

\smallskip\noindent
\emph{Polynomial $L^4$ trace on facets.}
There exists $C_{tr,4}>0$, depending only on shape-regularity and $k$, such that for every facet $F\subset\partial T$ and every
$z_h\in[\IP^k(T)]^d$,
\begin{equation}\label{eq:L4_trace_poly}
\|z_h\|_{L^4(F)} \le C_{tr,4}\,h_T^{-\frac14}\,\|z_h\|_{L^4(T)}.
\end{equation}

\smallskip\noindent
\emph{Broken $W^{1,4}$ trace on facets.}
There exists $C_{W,4}>0$, depending only on shape-regularity, such that for every facet
$F\subset\partial T$ and every $w\in [W^{1,4}(T)]^d$,
\begin{equation}\label{eq:w14_trace}
\|w\|_{L^4(F)}
\le
C_{W,4}\Bigl(
h_T^{-\frac14}\|w\|_{L^4(T)}
+
h_T^{\frac34}\|\nabla w\|_{L^4(T)}
\Bigr).
\end{equation}

\section{Possible choices for \texorpdfstring{$\blfconv{w}$}{the convection term}}\label{sec:choices_ch}

The two conservative variants are only needed later in the polynomial case, so we keep their discussion in that setting.
We record three standard DG realizations of the convective term.
All of them fit into \cref{ass:ch} (with possibly $\blfconv{w}(u_h,u_h)=0$ or, after adding upwinding, $\blfconv{w}(u_h,u_h)\ge 0$), and they differ mainly in whether they are locally conservative and in how the kinetic-energy balance is enforced.

\subsection{Temam-type (energy-stable) form}
Following \cite{DE10,DiPietroErn}, a discrete counterpart of Temam's modification is obtained by adding a divergence correction
\begin{align}\label{eq:ch_temam}
\blfconv{w}(u,v)
&:= (w\cdot\nabla u,  v)_\Omega
- \big(\avg{w}\cdot n_F \jmp{u}, \avg{v}\big)_\Fhi \nonumber\\
&\quad + \frac12\big((\nabla\cdot w)u,v\big)_\Omega
-\frac12\big(\jmp{w}\cdot n_F, \avg{u\cdot v}\big)_\Fh .
\end{align}
This choice satisfies the exact cancellation $\blfconv{v}(v,v)=0$, but it is not locally conservative since it contains a source term proportional to $\nabla\cdot w$.

\begin{lemma}\label{lem:ch_temam}
    The Temam-type trilinear form \eqref{eq:ch_temam} satisfies \cref{ass:ch}.
\end{lemma}
\begin{proof}
We first prove \eqref{eq:ch_mixed_L4}.
No Trefftz property of the second argument is used below.
We use the skew-symmetry in the last two arguments and therefore bound
\[
\blfconv{w}(v_h,z)=I+II+III+IV,
\]
with the four terms corresponding to the four terms in \eqref{eq:ch_temam}.

For the volume term, by H\"older's inequality, \eqref{l4}, and the elementwise
Poincar\'e inequality for $z\in\IX_{\rm ps}(\Th)$,
\[
|I|
\le
\|w\|_{L^4(\Th)}\|\nabla v_h\|_{\Th}\|z\|_{L^4(\Th)}
\lesssim
\nu |w|_{h,d}\,\|v_h\|_{1,h}\,\|z\|_{1,h}.
\]

From the standard trace theorem from $H^1(\widehat T)\to H^{1/2}(\widehat F)$ on the reference element and scaling we have
\[
\|z\|_{L^2(F)}
+
h_T^{1/2}|z|_{H^{1/2}(F)}
\lesssim
h_T^{-1/2}\|z\|_{L^2(T)}
+
h_T^{1/2}\|\nabla z\|_{L^2(T)}\quad
\forall\,z\in H^1(T). 
\]
Using the Sobolev embedding $H^{1/2}(F)\hookrightarrow L^4(F)$, valid for $d=2,3$, with the corresponding scaling, we obtain
\[
\|h_T^{1/4}z\|_{L^4(F)}
\lesssim
h_T^{-d/4}\|z\|_{L^2(T)}
+
h_T^{1-d/4}\|\nabla z\|_{L^2(T)} .
\]
Using the Poincar\'e, as  $z\in\IX_{\rm ps}(\Th)$, we obtain
\begin{equation}\label{eq:IXps_h14_facet}
\|h^{\frac14} z\|_{L^4(F)}
\le C h_T^{1-\frac d4}\|\nabla z\|_{T}.
\end{equation}
Moreover, by \eqref{eq:w14_trace} and finite overlap of elements
\[
\|h^{\frac14}\avg{w}\|_{L^4(\Fhi)}
\lesssim
\|w\|_{L^4(\Th)}+\|h\nabla w\|_{L^4(\Th)}
\lesssim
\nu |w|_{h,d}.
\]
Hence, for the facet transport term,
\[
|II|
\le
\|h^{\frac14}\avg{w}\|_{L^4(\Fhi)}
\|h^{-\frac12}\jmp{v_h}\|_{L^2(\Fhi)}
\|h^{\frac14}\avg{z}\|_{L^4(\Fhi)}
\lesssim
\nu |w|_{h,d}\,\|v_h\|_{1,h}\,\|z\|_{1,h}.
\]

For the divergence term, using standard embedding 
$
\|h_T^{1-\frac d4}\Div w\|_{L^2(T)}
\lesssim
\|h_T^{2-\frac d4}\nabla w\|_{L^4(T)}
$
together with \eqref{l4} and the Poincar\'e inequality for $z$, we obtain
\[
|III|
\le
\|\Div w\|_{\Th}\|v_h\|_{L^4(\Th)}\|z\|_{L^4(\Th)}
\lesssim
\nu |w|_{h,d}\,\|v_h\|_{1,h}\,\|z\|_{1,h}.
\]

Finally, using
$
\|\avg{z\cdot v_h}\|_{L^{\frac43}(F)}
\lesssim
\|z\|_{L^2(F)}\|v_h\|_{L^4(F)},
$
the trace inequality for $z$, \eqref{eq:L4_trace_poly} for $v_h$, and the jump term in
$|w|_{h,d}$, we get
\begin{align*}
|IV|
&\lesssim
\|\jmp{w}\cdot n\|_{L^4(\Fh)}
\|z\|_{L^2(\Fh)}
\|v_h\|_{L^4(\Fh)}
\lesssim
\bigl(\nu h^{-\frac14}|w|_{h,d}\bigr)
\bigl(h^{\frac12}\|\nabla z\|_{\Th}\bigr)
\bigl(h^{-\frac14}\|v_h\|_{L^4(\Th)}\bigr)
\\
&\lesssim
\nu |w|_{h,d}\,\|z\|_{1,h}\,\|v_h\|_{1,h}.
\end{align*}

Combining the bounds for $I,\dots,IV$ proves \eqref{eq:ch_mixed_L4}.

The proof of \eqref{eq:ch_cont} is the same. 
Decomposing into  the four terms as above, following the same steps, but instead of elementwise Poincar\'e and \eqref{l4} for $z$, use the discrete Sobolev embedding for $z\in\IP^k$, not gaining any powers of $h$.
\end{proof}

\subsection{Conservative form via a modified pressure}
A locally conservative alternative in \cite{DE10} is based on rewriting the momentum equation using a modified pressure so that the convective contribution is in divergence form.
One discrete realization is
\begin{align}\label{eq:ch_cons_pressure}
\blfconv{w_h}(u_h,v_h)
&:= -\big(w_{h} , (u_h\cdot\nabla) v_{h}\big)_\Omega
+ \big((n_F\cdot\avg{u_h}) \avg{w_{h}}, \jmp{v_{h}}\big)_\Fhi \nonumber\\
&\quad + \frac12\big(v_h, \nabla(u_{h}\cdot w_{h})\big)_\Omega
-\frac12\big(n_F\cdot\avg{v_h}, \jmp{u_{h}\cdot w_{h}}\big)_\Fhi .
\end{align}
The salient feature is local conservativity (in the standard DG flux sense), while still permitting an energy estimate through the same abstract assumptions on $\blfconv{w}$ used in \cite{DE10}. 
Using similar arguments as for the Temam-type form, we can show that \eqref{eq:ch_cons_pressure} also satisfies \cref{ass:ch}.

\subsection{Upwind conservative form with exactly solenoidal convecting field}
In the LDG framework of \cite{CKS05}, convection is discretized by an upwind numerical flux.
In the mixed (primal) setting, the convective contribution can be written (for a given convecting field $w$) as
\begin{align} \label{eq:ch_CKS_upwind}
    \blfconv{w}(u_h,v_h) &:= 
            -((\calP_{\Div} w_h\cdot\nabla)u_h,v_h)_\Th
            +\big\langle (\calP_{\Div} w_h\cdot n_F) u_h^w,\jmp{v_h}\big\rangle_{\Fhi}
            +\big\langle (\calP_{\Div} w_h\cdot n)\,u_h^{w}, v_h\big\rangle_{\Fhb}
\end{align}
where $\calP_{\Div}$ is a suitable projection onto the space of $H(\Div)$-conforming element-wise divergence-free vector fields, for example the BDM-interpolation operator. 
And $u_h^{w}$ is the standard upwind trace
\[
    u_h^{w} = \begin{cases}
        u_h^+ & \text{if } \calP_{\Div}w\cdot n_F > 0,\\
        u_h^- & \text{if } \calP_{\Div}w\cdot n_F < 0.
    \end{cases}
\]
A key observation in \cite{CKS05} is that local conservativity follows provided the convecting field is globally divergence-free (in $H(\Div)$),
and they enforce this by taking $w = P u_h$ where $P$ is an element-wise post-processing/projection producing an exactly solenoidal field. 

Following the same approach as in the proof of \cref{lem:ch_temam}, we can show that the upwind form \eqref{eq:ch_CKS_upwind} also satisfies the abstract assumptions in \cref{ass:ch}.
\begin{lemma}\label{lem:th_upwind_mixed_IL}
Let $\blfconv{w}$ be given by \eqref{eq:ch_CKS_upwind}.
Assume $w\in\IW$ and set $\calP_{\Div}w_h\in H(\Div;\Omega)\cap[\IP^k]^d$ with $\Div \tilde w_h|_T=0$ for all $T\in\Th$.
Assume moreover that $\calP_{\Div}$ is $L^4$-stable in the sense that
\begin{equation*}
\|\calP_{\Div} w_h\|_{L^4(T)}\le C_{\Div}\,\|w_h\|_{L^4(T)}\qquad\forall T\in\Th.
\end{equation*}
Then \eqref{eq:ch_mixed_L4} holds.
\end{lemma}

\bibliographystyle{abbrv}
\bibliography{bib}

\end{document}